\documentclass{amsart}

\usepackage{amsfonts,amsmath,amsthm,thmtools,thm-restate, amssymb, mathtools,verbatim,textcomp,amscd}
\usepackage{latexsym,epsfig,verbatim}
\usepackage{graphics,textcomp,  comment}
\usepackage{graphicx}
\usepackage{color}
\usepackage{url}
\usepackage{psfrag}
\usepackage[utf8]{inputenc}

\usepackage{amssymb}
\usepackage{nopageno}
\usepackage{amsmath,amsthm,thm-restate,thmtools,amsfonts,amssymb}
\usepackage{ifxetex}
\usepackage{enumitem, multicol}

\usepackage{amsmath,amsthm, amssymb, bm, bigints, upgreek}    
\usepackage{graphicx, verbatim, color,  url}   
\usepackage[usenames,dvipsnames]{xcolor} 

\begin{document}

	\newtheorem{theorem}{Theorem}[section]
	\newtheorem{prop}[theorem]{Proposition}
	\newtheorem{lemma}[theorem]{Lemma}
	\newtheorem{cor}[theorem]{Corollary}
	\newtheorem{definition}[theorem]{Definition}
	\newtheorem{conj}[theorem]{Conjecture}
	\newtheorem{claim}[theorem]{Claim}
	\newtheorem{qn}[theorem]{Question}
	\newtheorem{defn}[theorem]{Definition}
	\newtheorem{defth}[theorem]{Definition-Theorem}
	\newtheorem{obs}[theorem]{Observation}
	\newtheorem{rmk}[theorem]{Remark}
	\newtheorem{ans}[theorem]{Answers}
	\newtheorem{slogan}[theorem]{Slogan}
	
	\newtheorem{sublem}[theorem]{Sublemma}
	
	\newtheorem{fact}[theorem]{Fact}
	\newtheorem*{fact*}{Fact}

	\newtheorem{notn}[theorem]{Notation}
	\newtheorem{eg}[theorem]{Example}

	\newcommand{\boundary}{\partial}
	\newcommand{\hhat}{\widehat}
	\newcommand{\C}{{\mathbb C}}
	\newcommand{\Ga}{{\Gamma}}
	\newcommand{\G}{{\Gamma}}
	\newcommand{\s}{{\Sigma}}
	\newcommand{\PSL}{{PSL_2 (\mathbb{C})}}
	\newcommand{\pslc}{{PSL_2 (\mathbb{C})}}
	\newcommand{\pslr}{{PSL_2 (\mathbb{R})}}
	\newcommand{\Gr}{{\mathcal G}}
	\newcommand{\integers}{{\mathbb Z}}
	\newcommand{\natls}{{\mathbb N}}
	\newcommand{\ratls}{{\mathbb Q}}
	\newcommand{\reals}{{\mathbb R}}
	\newcommand{\proj}{{\mathbb P}}
	\newcommand{\lhp}{{\mathbb L}}
	\newcommand{\tube}{{\mathbb T}}
	\newcommand{\cusp}{{\mathbb P}}
	\newcommand\AAA{{\mathcal A}}
	\newcommand\HHH{{\mathbb H}}
	\newcommand\BB{{\mathcal B}}
	\newcommand\CC{{\mathcal C}}
	\newcommand\DD{{\mathcal D}}
	\newcommand\EE{{\mathcal E}}
	\newcommand\FF{{\mathcal F}}
	\newcommand\GG{{\mathcal G}}
	\newcommand\HH{{\mathcal H}}
	\newcommand\II{{\mathcal I}}
	\newcommand\JJ{{\mathcal J}}
	\newcommand\KK{{\mathcal K}}
	\newcommand\LL{{\mathcal L}}
	\newcommand\MM{{\mathcal M}}
	\newcommand\NN{{\mathcal N}}
	\newcommand\OO{{\mathcal O}}
	\newcommand\PP{{\mathcal P}}
	\newcommand\QQ{{\mathcal Q}}
	\newcommand\RR{{\mathcal R}}
	\newcommand\SSS{{\mathcal S}}
	\newcommand\TT{{\mathcal T}}
	\newcommand\UU{{\mathcal U}}
	\newcommand\VV{{\mathcal V}}
	\newcommand\WW{{\mathcal W}}
	\newcommand\XX{{\mathcal X}}
	\newcommand\YY{{\mathcal Y}}
	\newcommand\ZZ{{\mathcal Z}}
	\newcommand\CH{{\CC\Hyp}}
	\newcommand{\Chat}{{\hat {\mathbb C}}}
	\newcommand\MF{{\MM\FF}}
	\newcommand\PMF{{\PP\kern-2pt\MM\FF}}
	\newcommand\ML{{\MM\LL}}
	\newcommand\PML{{\PP\kern-2pt\MM\LL}}
	\newcommand\GL{{\GG\LL}}
	\newcommand\Pol{{\mathcal P}}
	\newcommand\half{{\textstyle{\frac12}}}
	\newcommand\Half{{\frac12}}
	\newcommand\Mod{\operatorname{Mod}}
	\newcommand\Area{\operatorname{Area}}
	\newcommand\ep{\epsilon}
	\newcommand\Hypat{\widehat}
	\newcommand\Proj{{\mathbf P}}
	\newcommand\U{{\mathbf U}}
	\newcommand\Hyp{{\mathbf H}}
	\newcommand\D{{\mathbf D}}
	\newcommand\Z{{\mathbb Z}}
	\newcommand\R{{\mathbb R}}
	\newcommand\Q{{\mathbb Q}}
	\newcommand\E{{\mathbb E}}
	\newcommand\EXH{{ \EE (X, \HH_X )}}
	\newcommand\EYH{{ \EE (Y, \HH_Y )}}
	\newcommand\GXH{{ \GG (X, \HH_X )}}
	\newcommand\GYH{{ \GG (Y, \HH_Y )}}
	\newcommand\ATF{{ \AAA \TT \FF }}
	\newcommand\PEX{{\PP\EE  (X, \HH , \GG , \LL )}}
	\newcommand{\lct}{\Lambda_{CT}}
	\newcommand{\lel}{\Lambda_{EL}}
	\newcommand{\lgel}{\Lambda_{GEL}}
	\newcommand{\lre}{\Lambda_{\mathbb{R}}}

	\newcommand\til{\widetilde}
	\newcommand\length{\operatorname{length}}
	\newcommand\tr{\operatorname{tr}}
	\newcommand\gesim{\succ}
	\newcommand\lesim{\prec}
	\newcommand\simle{\lesim}
	\newcommand\simge{\gesim}
	\newcommand{\simmult}{\asymp}
	\newcommand{\simadd}{\mathrel{\overset{\text{\tiny $+$}}{\sim}}}
	\newcommand{\ssm}{\setminus}
	\newcommand{\diam}{\operatorname{diam}}
	\newcommand{\pair}[1]{\langle #1\rangle}
	\newcommand{\T}{{\mathbf T}}
	\newcommand{\inj}{\operatorname{inj}}
	\newcommand{\pleat}{\operatorname{\mathbf{pleat}}}
	\newcommand{\short}{\operatorname{\mathbf{short}}}
	\newcommand{\vertices}{\operatorname{vert}}
	\newcommand{\collar}{\operatorname{\mathbf{collar}}}
	\newcommand{\bcollar}{\operatorname{\overline{\mathbf{collar}}}}
	\newcommand{\I}{{\mathbf I}}
	\newcommand{\tprec}{\prec_t}
	\newcommand{\fprec}{\prec_f}
	\newcommand{\bprec}{\prec_b}
	\newcommand{\pprec}{\prec_p}
	\newcommand{\ppreceq}{\preceq_p}
	\newcommand{\sprec}{\prec_s}
	\newcommand{\cpreceq}{\preceq_c}
	\newcommand{\cprec}{\prec_c}
	\newcommand{\topprec}{\prec_{\rm top}}
	\newcommand{\Topprec}{\prec_{\rm TOP}}
	\newcommand{\fsub}{\mathrel{\scriptstyle\searrow}}
	\newcommand{\bsub}{\mathrel{\scriptstyle\swarrow}}
	\newcommand{\fsubd}{\mathrel{{\scriptstyle\searrow}\kern-1ex^d\kern0.5ex}}
	\newcommand{\bsubd}{\mathrel{{\scriptstyle\swarrow}\kern-1.6ex^d\kern0.8ex}}
	\newcommand{\fsubeq}{\mathrel{\raise-.7ex\hbox{$\overset{\searrow}{=}$}}}
	\newcommand{\bsubeq}{\mathrel{\raise-.7ex\hbox{$\overset{\swarrow}{=}$}}}
	\newcommand{\tw}{\operatorname{tw}}
	\newcommand{\base}{\operatorname{base}}
	\newcommand{\trans}{\operatorname{trans}}
	\newcommand{\rest}{|_}
	\newcommand{\bbar}{\overline}
	\newcommand{\UML}{\operatorname{\UU\MM\LL}}
	\newcommand{\EL}{\mathcal{EL}}
	\newcommand{\qle}{\lesssim}

	\newcommand\Gomega{\Omega_\Gamma}
	\newcommand\nomega{\omega_\nu}

\newcommand\pref[1]{(\ref{#1})}
\newcommand{\Tau}{\mathrm{T}}

\newcommand{\out}{\mathsf{Out(\mathbb{F})}}
\newcommand{\F}{\ensuremath{\mathbb{F} } }
\newcommand{\aut}{\mathsf{Aut(\mathbb{F})}}

\title{Regluing graphs of Free Groups}

\author{Pritam Ghosh}
\address{Department of Mathematics, Ashoka University, Haryana, India}
\email{pritam.ghosh@ashoka.edu.in}

\author{Mahan Mj}
\address{School of Mathematics, Tata Institute of Fundamental Research, 1 Homi Bhabha Road, Mumbai 400005, India}

\email{mahan@math.tifr.res.in}
\email{mahan.mj@gmail.com}
\urladdr{http://www.math.tifr.res.in/~mahan}

\thanks{MM is   supported by  the Department of Atomic Energy, Government of India, under project no.12-R\&D-TFR-14001.
	MM is also supported in part by a Department of Science and Technology JC Bose Fellowship, CEFIPRA  project No. 5801-1, a SERB grant MTR/2017/000513, and an endowment of the Infosys Foundation via the Chandrasekharan-Infosys Virtual Centre for Random Geometry.   This material is based upon work partially supported by the National Science Foundation
	under Grant No. DMS-1928930 while MM participated in a program hosted
	by the Mathematical Sciences Research Institute in Berkeley, California, during the
	Fall 2020 semester.} 
	
\thanks{P. Ghosh is supported by the faculty research grant of Ashoka university.}

\date{\today}

\begin{abstract}
	Answering a question due to Min, we prove that a finite
	 graph of roses admits a  regluing such that the resulting graph of roses has hyperbolic fundamental group.
\end{abstract}

\maketitle

\tableofcontents

\maketitle

\section{Introduction }\label{sec-intro}
Let $\GG$ be a finite graph and $\pi: \XX \to \GG$ be a finite graph of spaces where each vertex and edge space is a finite graph and the edge-to-vertex maps are homotopic to covering maps of finite degree. We call such a graph of spaces a 
homogeneous graph of roses. 
Cutting along the edge graphs and pre-composing one of the resulting attaching maps
by homotopy equivalences inducing  hyperbolic automorphisms of the corresponding edge groups,we obtain a hyperbolic regluing of  $\pi: \XX \to \GG$, the initial homogeneous graph  of roses
(see Section \ref{sec-reglue} for more precise details).
A consequence of the main theorem of this paper is:

\begin{theorem}\label{main-intro}
Given a homogeneous graph of roses, there exist hyperbolic regluings such that the resulting graph of spaces has hyperbolic fundamental group.
\end{theorem}

Theorem \ref{main-intro} answers a question due to Min \cite{min},
who proved the analogous theorem for homogeneous graphs of hyperbolic surface groups. The main theorem of this paper (see Theorem \ref{thm-main}) identifies precise conditions under which
the conclusions of Theorem \ref{main-intro} hold.
Min's theorem built on and generalized work of Mosher \cite{Mos-97}, who proved the existence of surface-by-free hyperbolic groups. An analogous theorem, proving the existence of free-by-free hyperbolic groups, is due to Bestvina, Feighn and Handel \cite{BFH-97}. This last theorem from 
\cite{BFH-97} can be recast in the framework of Theorem 
\ref{main-intro} by demanding, in addition, that all edge-to-vertex inclusions for a homogeneous graph of roses are
homotopy equivalences.
Theorem 
\ref{main-intro} generalizes this theorem by relaxing the hypothesis on edge-to-vertex inclusion maps, and allowing them  to be homotopic to finite degree covers. 

Theorem \ref{main-intro} also furnishes new examples of metric bundles in the sense of Mj-Sardar \cite{mahan-sardar}, where all vertex and edge spaces are trees and thus examples to which the results in \cite{mahan-sardar2} applies. A basic question resulting from \cite{min} and the present paper is the following:

\begin{qn}\label{qn-el}
Develop a theory of ending laminations for homogeneous graphs of surfaces
and a similar one for roses.
\end{qn}

A rich theory of ending laminations was developed for Kleinian surface
groups \cite{thurstonnotes} concluding with the celebrated ending lamination
theorem \cite{minsky-elc2}. A theory oriented towards hyperbolic group extensions was developed in \cite{mitra-endlam} and some consequences
derived in \cite{mahan-rafi}. The intent of Question \ref{qn-el} is to ask for an analogous theory in the context of homogeneous graphs of spaces.

\subsection{Regluing}\label{sec-reglue}
We refer to \cite{scott-wall} for generalities on graphs and trees of spaces. A word about the notational convention we shall follow. We shall use $\GG$ to denote the base graph in a graph of spaces, and $G$ to denote a graph whose self-homotopy equivalence classes give $\out$. The vertex (resp.\ edge) set of $\GG$ will be denoted as $V(\GG)$ (resp. $E(\GG)$).

\begin{defn}\cite{BF}(Graphs of hyperbolic  spaces with qi condition)\label{def-tree}
Let $\GG$ be a graph (finite or infinite), and $\XX$  a 
geodesic metric space. Then a triple
$(\XX, \GG, \pi)$
with $\pi:\XX\to \GG$ is called a graph
 of hyperbolic metric spaces
	with qi embedded condition if there exist $\delta\geq 0$, $K\geq 1$  such that: 
	\begin{enumerate}
	\item For all $v\in V(\GG)$, $\XX_v=\pi^{-1}(v)$ is $\delta-$hyperbolic   with respect to the path metric $d_v$, induced from
	$\XX$. Further,  the inclusion maps $\XX_v \to \XX$ are uniformly  proper.
	\item Let $e=[v,w]$ be an edge of $\GG$ joining $v,w\in V(\GG)$. Let $m_e\in \GG$ be the midpoint of $e$.
	Then $\XX_e=\pi^{-1}(m_e)$ is $\delta-$hyperbolic   with respect to the path metric $d_e$, induced from
	$\XX$. The pre-image $\pi^{-1}((v,w))$ is identified
	with $\XX_e \times ((v,w))$.
	\item  The  attaching maps $\psi_{e,v}$ (resp. $\psi_{e,w}$) from $\XX_e\times \{v\}$ (resp.
	$\XX_e\times \{w\}$) are $K$-qi embeddings to $(\XX_v,d_v)$ (resp. $(\XX_w,d_w)$).
	\end{enumerate}
\end{defn}

Throughout this paper, we shall be interested in the following
special cases of graphs of hyperbolic spaces:

\begin{enumerate}
\item $\GG$ is a finite graph, each $X_v, X_e$ is a finite graph,
and each $\psi_e: \XX_e \to \XX_v$ induces an injective map $\psi_{e*}:\pi_1( \XX_e) \to \pi_1(\XX_v)$ at the level of fundamental groups such that $[\pi_1(\XX_v): \psi_{e*}(\pi_1( \XX_e))]$ is finite. We shall call such a graph of spaces a {\it homogeneous graph of roses}.
\item The universal cover of a  homogeneous graph of roses
 yields a tree of spaces
such that all vertex and edge spaces are locally finite trees,
and edge-to-vertex space inclusions are quasi-isometries.
We shall call such a tree of spaces a {\it homogeneous tree of trees}.
\end{enumerate}

Let $\Pi:\YY \to \TT$ be a homogeneous  tree of trees arising as the universal cover of a  homogeneous graph of roses $\pi: \XX \to \GG$.

\begin{defn}\label{defnofhallway}\cite{BF}
	A disk $f : [-m,m]{\times}{I} \rightarrow 
	\YY$ is a \emph{ hallway} of length $2m$ if it satisfies the following conditions:\\
1) $f^{-1} ({\cup}{X_v} : v \in V(\TT)) = \{-m,  \cdots , m \}{\times}
		I$\\
2)  $f$ maps $i{\times}I$ to a geodesic in some  $(X_v,d_v)$.
3) $f$ is transverse, relative to condition (1) to the union $\cup_e X_e$.
\end{defn}

\begin{defn}\label{defnofrhothin}\cite{BF} A hallway $f : [-m,m]{\times}{I} \rightarrow 
	\YY$ is \emph{$\rho$-thin} if 
	$d({f(i,t)},{f({i+1},t)}) \leq \rho$ for all $i, t$.
	
	A hallway $f : [-m,m]{\times}{I} \rightarrow 
	X$ is said to be \emph{ $\lambda$-hyperbolic}  if 
	$$\lambda l(f(\{ 0 \} \times I)) \leq \, {\rm max} \ \{ l(f(\{ -m \} \times I)),
	l(f(\{ m \} \times I)).$$

	The quantity ${\rm min_i} \, \{ l(f(\{ i \} \times I))\}$ is called the \emph{ girth} of the hallway.

	A hallway is \emph{ essential} if the edge path in $T$ 
	resulting from projecting the hallway under $P\circ f$
	onto $T$ does not backtrack (and is therefore a geodesic segment in
	the tree $T$).
\end{defn}

\begin{defn}[Hallways flare condition]\cite{BF}\label{defnofflare} 
	The tree of spaces, $X$, is said to satisfy the \emph{ hallways flare}
	condition if there are numbers $\lambda > 1$ and $m \geq 1$ such that
	for all $\rho$ there is a constant $H:=H(\rho )$ such that  any
	$\rho$-thin essential hallway of length $2m$ and girth at least $H$ is
	$\lambda$-hyperbolic. In general, $\lambda, m$ will be called the constants of the hallways flare condition. 
\end{defn}

We now describe a process of regluing by adapting Min's notion
of graph of surfaces with pseudo-Anosov regluing \cite[p. 450]{min}.\\

\noindent {\bf Hyperbolic Regluing of a homogeneous graph  of roses:}
A homogeneous graph $\pi: \XX \to \GG$ of roses can be subdivided canonically by introducing vertices corresponding to mid-points of edges in $\GG$, so that
each edge in $\GG$ is now subdivided into two edges. Let $\GG(m)$
denote the subdivided graph.
Each
such new vertex is called a {\it mid-edge vertex}. The 
mid-edge vertex corresponding to $[v,w]$ is denoted as $m([v,w])$
and
the corresponding vertex space by $X_{mvw}$.
If the gluing maps corresponding to the new edge-to-vertex 
inclusions are taken to be the identity, then we obtain a new graph of spaces $\pi: \XX \to \GG(m)$ whose total space is homeomorphic to (and hence identified canonically with) $\XX$ and
$\pi$ is the same as before; only the simplicial structure
of $\GG$ has changed to $\GG(m)$. These maps
are called the {\it mid-edge inclusions}. 
\begin{defn}\label{def-reglue}
For each edge $e$ of $\GG(m)$, changing one of the mid-edge inclusions by a map $\phi_e$ representing an automorphism 
$\phi_{e*}$ of $\pi_1(X_e)$ gives a new graph of spaces 
$\pi_{reg}: \XX_{reg} \stackrel{\{\phi_e\}} \longrightarrow \GG$
called a {\it regluing of  $\pi: \XX \to \GG$ }
corresponding to the tuple $\{\phi_e\}$. 

If the universal cover $\til \XX_{reg}$ is hyperbolic,
we say that $\pi_{reg}: \XX_{reg} \stackrel{\{\phi_e\}} \longrightarrow \GG$ is a hyperbolic regluing of $\pi: \XX \to \GG$.
\end{defn}
We denote such a 
regluing by $(\XX_{reg},\GG,\pi, \{\phi_e\} )$. Let 
 $(\til\XX_{reg},\TT,\pi_{reg}, \{\til \phi_e\} )$ denote the universal cover
 of such a regluing.  
Note that the mid-edge inclusions in  $(\til\XX_{reg},\TT,\pi_{reg}, \{\til \phi_e\} )$ corresponding to lifts of the edge $e$ are given by lifts 
$\til \phi_e$ of $\phi_e$, and hence are $K(e)-$quasi-isometries,
where $K(e)$ depends on $\phi_e$.

We shall define an independent family of automorphisms precisely later (Definition \ref{ind}). For now, we say that
two hyperbolic automorphisms $\phi_1, \phi_2$ labeling a pair of edges $e_1, e_2$ incident on a vertex $v$ are independent,
if for the four sets of stable and unstable laminations that 
$\phi_1, \phi_2$ define,
no leaf of  any set is asymptotic to the leaf of another set.
Further, we demand that this condition is satisfied even after translation of laminations by distinct coset representatives of the edge group in the vertex group. A regluing where 
automorphisms labeling any pair of edges $e_1, e_2$ incident on a vertex $v$ are independent is called an independent
regluing.
We can now state the main Theorem of this paper
(see Theorem \ref{thm-main}) which is  a stronger version of Theorem \ref{main-intro}:

\begin{theorem}\label{thm-main-intro} Let $\pi:\XX \to \GG$ be a homogeneous graph of roses, and let
	$\{\phi_e\}, e \in E(\GG)$ be a tuple of hyperbolic automorphisms such that $(\XX_{reg},\GG,\pi, \{\phi_e\} )$ is an independent regluing. 
	Then there exist $k, n \in \natls$
	such that 
	$(\XX_{reg},\GG,\pi, \{\phi_e^{km_e}\} )$
	gives a hyperbolic rotationless regluing
	for all $m_e \geq n$. .
\end{theorem}

\section{Preliminaries on $\out$}\label{sec-out}
In this section we give the reader a short review of the definitions and some important results in $\out$ that are relevant to this paper. For details, see 
\cite{BFH-00}, \cite{FH-11}, \cite{HM-09}, \cite{HM-20}.  We  fix a hyperbolic $\phi\in\out$ for the purposes of this section. 

 A \textit{marked graph} is a finite graph $G$ which has no valence 1 vertices and is equipped with 
 a homotopy equivalence, called a {\it marking}, to the rose $R_n$
 given by $\rho: G\to R_n$ (where $n = \text{rank}(\F)$). The homotopy inverse of the marking is denoted by the map $\overline{\rho}: R_n \to G$. A \textit{circuit} in a marked graph is an immersion (i.e. a locally injective  continuous map) of $S^1$ into $G$. $I$ will denote an interval in $\R$ that is closed as a subset.
 A \emph{path} is a locally injective, continuous map $\alpha: I \to G$, 
  such that any lift $\til{\alpha}: I \to \til{G}$ is proper.
  When $I$ is compact, any continuous map from $I$ can be homotoped relative to  its endpoints by a process called \textit{tightening}  to a unique path (up to reparametrization)  with domain $I$. 
  If $I$ is noncompact then each lift $\til{\alpha}$  induces an injection from the ends of $I$ to the ends of $\til{G}$. In this case
  there is a unique path $[\alpha]$ which is homotopic to $\alpha$ such that both $[\alpha]$ and $\alpha$ have lifts to $\widetilde{G}$ with the 
  same finite endpoints and the same infinite ends. If $I$ has two infinite ends then $\alpha$ is called a \textit{line} in $G$ otherwise if $I$ has only one infinite end then $\alpha$ is called a \emph{ray}.
     Given a homotopy equivalence of marked graphs $f: G\to G'$,  $f_\#(\alpha)$ denotes the tightened image $[f(\alpha)]$ in $G'$. Similarly we define $\til{f}_\#(\til{\alpha})$ by lifting to the universal cover. 
 
 A \textit{topological representative} of $\phi$ is a homotopy equivalence $f:G\rightarrow G$ such that $\rho: G \rightarrow R_n$ is a marked graph, 
$f$ takes vertices to vertices and edges to edge-paths and the map $\rho\circ f \circ \overline{\rho}: R_n \rightarrow R_n$ induces the outer automorphism $\phi$ at the level of fundamental groups. A nontrivial path $\gamma$ in $G$ 
is a \textit{periodic Nielsen path} if there exists a $k$ such that $f^k_\#(\gamma)=\gamma$;
the minimal such $k$ is called the period. If $k=1$, we simply
call such a path \textit{Nielsen path}. A periodic Nielsen path is \textit{indivisible} if it cannot 
be written as a concatenation of two or more nontrivial periodic Nielsen paths.

\noindent \textbf{Filtrations and legal paths: } Given a subgraph $H\subset G$ let $G\setminus H$ denote the union of edges in $G$ that are not in $H$.
A \textit{filtration} of $G$ is a strictly increasing sequence of subgraphs 
 $G_0 \subset G_1 \subset \cdots \subset G_n = G$, each with no isolated vertices. 
 The individual terms $G_k$ are called \textit{filtration elements}, and if $G_k$ is a core graph (i.e. a graph without valence 1 vertices)
 then it is called a 
 \textit{core filtration element}. The subgraph $H_k = G_k \setminus G_{k-1}$ together with the vertices which occur as endpoints of edges in 
 $H_k$ is called the \textit{stratum of height $k$}.
 The \textit{height}\index{height} of a subset of $G$ is the minimum $k$ such that the subset is contained in $G_k$. 
 A \textit{connecting path} of a stratum $H_k$ is a nontrivial finite path $\gamma$ of height $< k$ whose endpoints are contained in 
 $H_k$.

Given a topological representative $f:G\rightarrow G$, one can define a map $T_f$ by setting $T_f(E)$ 
to be the first edge of the edge path  $f(E)$. We say $T_f(E)$ is the \emph{direction} of $f(E)$. If $E_1, E_2$ are two edges in $G$ with the same initial vertex, then the unordered pair $(E_1, E_2)$ is called a \textit{turn} in $G$. 
Define $T_f(E_1,E_2) = (T_f(E_1),T_f(E_2))$. 
So $T_f$ is a map that takes turns to turns. We say that a 
nondegenerate turn (i.e. $E_1\neq E_2$) is \textit{illegal} if for some $k>0$ the turn $T^k_f(E_1, E_2)$ becomes degenerate (i.e. $T^k_f(E_1) = T^k_f(E_2)$); otherwise the
 turn is \textit{legal}. A path is said to be a \textit{legal path} if it contains only legal turns. A path is $r-legal$ if it is of height $r$ and all its illegal turns are in $G_{r-1}$.
 We say that $f$ \textit{respects} the filtration or 
that the filtration is \textit{$f$-invariant} if $f(G_k) \subset G_k$ for all $k$.

\noindent \textbf{Weak topology:}
We define an equivalence relation on the set of all circuits and paths in $G$ by saying that two elements are equivalent if and only if they differ by some orientation preserving homeomorphism of their respective domains.
Let $\widehat{\mathcal{B}}(G)$, called the \emph{space of paths}, denote the  space of equivalence classes of circuits and paths in $G$, whose endpoints (if any) are vertices of $G$ . We give this space the \textit{weak topology}:
 for each finite path $\alpha$ in $G$,  the basic open set $\widehat{N}(G,\alpha)$  consists of all paths and circuits in $\widehat{\mathcal{B}}(G)$ which have $\alpha$ as a subpath. Then
 $\widehat{\mathcal{B}}(G)$ is compact in the weak topology.
Let $\mathcal{B}(G)\subset \widehat{\mathcal{B}}(G)$ be the compact subspace of all lines in $G$ with the induced topology: $\mathcal{B}(G)$ is called {\emph{space of lines}} of $G$.
One can give an equivalent description of $\mathcal{B}(G)$ following \cite{BFH-00}. 
A line is completely determined, up to reversal of direction, by two distinct points in $\partial \mathbb{F}$. 
Let $\widetilde{\mathcal{B}}=\{ \partial \mathbb{F} \times \partial \mathbb{F} - \Delta \}/(\mathbb{Z}_2)$, where $\Delta$ is the diagonal and $\mathbb{Z}_2$ acts by the flip. Equip 
$\widetilde{\mathcal{B}}$ with the  topology induced from the standard Cantor set topology on $\partial \mathbb{F}$. Then $\mathbb{F}$ acts on $\widetilde{\mathcal{B}}$ with a compact but non-Hausdorff quotient space $\mathcal{B}=\widetilde{\mathcal{B}}/\mathbb{F}$.
The quotient topology is also called the \textit{weak topology}
and it coincides with the topology defined in the previous paragraph. Elements of $\mathcal{B}$ are called \textit{lines}. A lift of a line $\gamma \in \mathcal{B}$ is an element  $\widetilde{\gamma}\in \widetilde{\mathcal{B}}$ that
projects to $\gamma$ under the quotient map and the two elements of $\partial\widetilde{\gamma}$ are called its \emph{endpoints} or simply \emph{ends}.
For any circuit $\alpha$, we take its ``infinite-fold concatenation" $\cdots\alpha.\alpha.\alpha\cdots$ and view it as a line. With this understanding,  we can talk of a circuit belonging to an open set $V\subset \mathcal{B}$.

An element $\gamma\in\mathcal{B}$ is said to be \textit{weakly attracted} to $\beta\in\mathcal{B}$ under the action of $\phi\in\out$, if some subsequence of $\{\phi^k(\gamma)\}_k$ converges  to $\beta$ in the weak topology as $k\to\infty$.
Similarly, if we have a homotopy equivalence $f:G\rightarrow G$,  a line(path) $\gamma\in\widehat{\mathcal{B}}(G)$ is said to be \textit{weakly attracted} to a line(path) $\beta\in\widehat{\mathcal{B}}(G)$ under the action of $f_{\#}$, 
if (some subsequence of) $\{f_{\#}^k(\gamma)\}_k$ converges to $\beta$
in the weak topology as $k\to\infty$. Note that since the space of paths and circuits is non-Hausdorff, a sequence can converge to multiple points in the space and any such point will be called a weak limit of the sequence.

The {\emph{accumulation set}} of a ray $\alpha$ in $G$ is the set of lines  $\ell\in \mathcal{B}$ which are elements of the weak closure of $\alpha$. This is equivalent to  saying that every finite subpath of $\ell$ occurs infinitely many times as a subpath of $\alpha$. Two rays are \emph{asymptotic} if they have equal subrays. This gives an equivalence relation on the set of all rays and two rays in the same equivalence class have the same closure. 
The weak accumulation set of some $\xi\in\partial\mathbb{F}/\F$ is the set of lines in the weak closure
of any  ray having end $\xi$. We call this the {\textit{weak closure}} of $\xi$.
    
\noindent  \textbf{Subgroup systems:}
Define a \textit{subgroup system} $\mathcal{A} = \{[H_1], [H_2], .... ,[H_k]\}$ to be a finite collection of distinct conjugacy classes of finite rank, nontrivial subgroups $H_i<\mathbb{F}$.
A subgroup system  is said to be a \emph{free factor system} if $\mathbb{F}$ has a free factor decomposition $\F = A_1 * A_ 2 * \cdots * A_k * B$, where $[H_i] = [A_i]$ for all $i$.  
A subgroup system $\mathcal{A}$ carries a conjugacy class $[c]\in \mathbb{F}$ if there exists some $[A]\in\mathcal{A}$ such that $c\in A$. Also, we say that $\mathcal{A}$ carries a line $\gamma$ if one of the following equivalent conditions hold:
\begin{itemize}
\item  $\gamma$ is the weak limit of a sequence of conjugacy classes carried by $\mathcal{A}$.
\item There exists some $[A]\in \mathcal{A}$ and a lift $\widetilde{\gamma}$ of $\gamma$ so that the endpoints of $\widetilde{\gamma}$ are in $\partial A$.

\end{itemize}

The {\emph{free factor support}} of a line $\ell$ in a marked graph $G$ is the conjugacy class of the minimal (with respect to inclusion) free factor of $\pi_1(G)$ which carries $\ell$. The existence of such a free factor is due to \cite[Corollary 2.6.5]{BFH-00}.
Let $\ell$ be any line in $G$.  Let the free factor support of $\ell$ be $[K]$. If $\mathcal{F}$ is any free factor system that carries $\ell$, then the minimality of $[K]$ ensures that there exist some $[A]\in \mathcal{F}$ such that $K < A$. In this case we say that the \emph{free factor support of }$\ell$ \emph{is carried by} $\mathcal{F}$.\\

\noindent\textbf{Attracting Laminations:}
For any marked graph $G$, the natural identification $\mathcal{B}\approx \mathcal{B}(G)$ induces a bijection between the closed subsets of $\mathcal{B}$ and the closed subsets of $\mathcal{B}(G)$. A closed
subset in either  case is called a \textit{lamination},
and is denoted by $\Lambda$. Given a lamination $\Lambda\subset \mathcal{B}$ we look at the corresponding lamination in $\mathcal{B}(G)$ as the
realization of $\Lambda$ in $G$. An element $\lambda\in \Lambda$ is called a \textit{leaf} of the lamination.
A lamination $\Lambda$ is called an \textit{attracting lamination} for a rotationless $\phi$ if it is the weak closure of a line $\ell$  such that 
\begin{enumerate}
	\item  $\ell$ is a birecurrent leaf of $\Lambda$.
	\item $\ell$ has an \textit{attracting neighborhood} $V$
	in the weak topology,  \emph{i.e.} $\phi(V)\subset V$; every line in $V$ is weakly attracted to $\ell$ under iteration by $\phi$; and $\{\phi^k(V)\mid k\geq 1\}$ is a neighborhood basis of $\ell$.
	\item no lift $\widetilde{\ell}\in \mathcal{B}$ of $\ell$ is the axis of a generator of a rank 1 free factor of $\mathbb{F}$ .
\end{enumerate}
Such an $\ell$ is called a {\emph{generic leaf}} of $\Lambda$. An attracting lamination of $\phi^{-1}$ is called a \emph{repelling lamination} of $\phi$.
The set of all attracting and repelling laminations of $\phi$ are denoted by  $\mathcal{L}^+_\phi$ and $\mathcal{L}^-_\phi$ respectively. \\

\noindent
\textbf{Attracting fixed points and principal lifts:}
The action of $\Phi\in\aut$ on $\F$ extends to the boundary and is denoted by $\widehat{\Phi}:\partial \mathbb{F}\rightarrow\partial \mathbb{F}$. Let Fix$(\widehat{\Phi})$ denote the set of fixed points of this action. We call an element $\xi$ of Fix$(\widehat{\Phi})$ an {\emph{attracting fixed point}} if there exists an open neighborhood $U\subset \partial \mathbb{F}$ of $\xi$ such that  $\widehat{\Phi}(U)\subset U$,
and for any point $Q\in U$ the sequence $\widehat{\Phi}^n(Q)$ converges to $\xi$. Let Fix$_+(\widehat{\Phi})$ denote the set of attracting fixed points of Fix$(\widehat{\Phi})$. Similarly let Fix$_-(\widehat{\Phi})$ denote the attracting fixed points of Fix$(\widehat{\Phi}^{-1})$. 
A lift $\Phi\in\aut$ is said to be \emph{principal} if Fix$_+(\widehat{\Phi})$  either has at least three points, or has two points which are not the endpoints of a lift of some generic leaf of an attracting
lamination belonging to  $\mathcal{L}^+_\phi$. The latter case appears only when we are dealing with reducible hyperbolic automorphisms which have superlinear \emph{NEG} edges (see below). It is not something that is present in the context of
mapping class groups. See \cite[Section 3.2]{FH-11} for more details. 
Set $\text{Fix}^+(\phi) = \bigcup\limits_{\Phi\in P(\phi)} \text{Fix}_+ (\widehat{\Phi})$, where $P(\phi)$  is the set of all principal lifts of $\phi$. 
 We define
 $\mathcal{B}_{\mathsf{Fix}^+}(\phi): = \bigcup_{\Phi\in P(\phi)} \{\ell\in\mathcal{B} \ \arrowvert \ \partial\til{\ell}\in\mathsf{Fix}_+(\widehat{\Phi})\}$. For a principal lift $\Phi$, the map $\widehat{\Phi}$ may have periodic points and we may miss out on some attracting fixed points. This is why we need to move to rotationless powers,
 where every periodic point of $\widehat{\Phi}$ becomes a fixed point (see \cite[Definition 3.13]{FH-11} for further details). A hyperbolic outer automorphism $\phi$ is said to be {\emph{rotationless}} if for every $\Phi\in P(\phi)$ and any $k\geq 1$, all attracting fixed points of $\widehat{\Phi}^k$ are attracting fixed points of $\widehat{\Phi}$ and the map $\Phi\to\Phi^k$ induces a bijection between $P(\phi)$ and  $P(\phi^k)$.
\begin{lemma}\cite[Lemma 4.43]{FH-11}
\label{rotationless}
 There exists a $K$ depending only upon the rank of the free group $\mathbb{F}$ such that for every $\phi\in \out$ , $\phi^K$ is rotationless.
\end{lemma}

\noindent\textbf{EG strata, NEG strata and Zero strata: }
Given an $f$-invariant filtration, for each stratum $H_k$ with edges $\{E_1,\ldots,E_m\}$, define
the \textit{transition matrix} of $H_k$ to be the square matrix whose $j^{\text{th}}$ column records the number of times 
$f(E_j)$ crosses the  edges
$\{E_1,\ldots,E_m\}$. If $M_k$ is the zero matrix then we say that $H_k$ is a \textit{zero stratum}. 
If $M_k$ irreducible --- meaning that for each $i,j$ there exists $p$ such that the $i,j$ entry of the $p^{\text{th}}$ 
power of the matrix is nonzero --- then we say that $H_k$ is \textit{irreducible}; and if one can furthermore choose $p$ independently 
of $i,j$ then we say that
$H_k$ is \textit{aperiodic}. Assuming that $H_k$ is irreducible, the Perron-Frobenius theorem gives the following: the matrix $M_k$ 
has a unique maximal eigenvalue $\lambda \ge 1$, called the \textit{Perron-Frobenius eigenvalue}, for which some associated eigenvector has 
positive entries: if $\lambda>1$ then we say that $H_k$ is an {\emph{exponentially growing}} or EG stratum; whereas if $\lambda=1$ 
then $H_k$ is a \emph{nonexponentially growing} or NEG stratum. 
If the lengths of the edges in a NEG stratum grow linearly under iteration by $f$ we say that the stratum has \emph{linear} growth. 
An NEG stratum that is neither fixed nor has linear growth is called  {\emph{superlinear}}. It is worth noting here that there are no linearly growing strata for hyperbolic outer automorphisms. 

An important result from \cite[Section 3]{BFH-00}  is that there is a bijection between exponentially growing strata and
attracting laminations, which implies that there are only finitely many elements in $\mathcal{L}^+_\phi$. The set 
 $\mathcal{L}^+_\phi$ is invariant under the action of $\phi$. 
 When it is nonempty, $\phi$ can permute the elements of $\mathcal{L}^+_\phi$ if $\phi$ is not rotationless. 
 For rotationless $\phi$, it is known that $\mathcal{L}^+_\phi$ is a fixed set \cite{FH-11}. 
 
\noindent
\textbf{Dual lamination pairs:}
Let  $\Lambda^+_\phi$ be an attracting lamination of $\phi$ and $\Lambda^-_\phi$ be an attracting lamination of $\phi^{-1}$. We say that this lamination pair is a \textit{dual lamination pair} if the free factor support of some (any) generic  leaf of $\Lambda^+_\phi$ is also the free factor support of some (any) generic leaf of $\Lambda^-_\phi$. 
By \cite[Lemma 3.2.4]{BFH-00}, there is
a bijection between $\mathcal{L}^+_\phi$ and $\mathcal{L}^-_\phi$ induced by this duality relation.
We denote a dual lamination pair $\Lambda^+_\phi, \Lambda^-_\phi$ of $\phi$ by $\Lambda^\pm_\phi$.

\noindent
 \textbf{Relative train track map:} Given a topological representative $f:G\rightarrow G$ with a 
 filtration $G_0\subset G_1\subset \cdot\cdot\cdot\subset G_n$ which is preserved by $f$, we say that $f$ is a  relative 
 train track map if the following conditions are satisfied for every EG stratum $H_r$:
 \begin{enumerate}
  \item $f$ maps r-legal paths to r-legal paths.
  \item If $\gamma$ is a nontrivial path in $G_{r-1}$ with its endpoints in $H_r$ then $f_\#(\gamma)$ has its end points in $H_r$.
  \item If $E$ is an edge in $H_r$ then $Tf(E)$ is an edge in $H_r$
 \end{enumerate}

 Suppose $\phi$ is hyperbolic and rotationless and $f:G\to G$ is a relative train-track map for $\phi$. Two periodic vertices are Nielsen equivalent if they are endpoints of some periodic Nielsen path in $G$. A periodic vertex $v$ 
 is a \emph{principal vertex} if  $v$ does not satisfy the condition that it is the only periodic vertex in its Nielsen equivalence class and that there are exactly two periodic directions at $v$, both of which are in the same $EG$ stratum. A {principal direction} in $G$ is a non-fixed, oriented edge $E$ whose initial vertex is principal and initial direction is fixed under iteration by $f$.\\

\noindent
 \textbf{Splittings:}  \cite{FH-11} Given a relative train track map $f:G\rightarrow G$, a splitting of a line, 
 path or a circuit $\gamma$ is a decomposition of $\gamma$ into subpaths $\cdots\gamma_0\gamma_1 \cdots\gamma_k\cdots  $ 
 such that for all $i\geq 1$, $f^i_\#(\gamma) =  \cdots f^i_\#(\gamma_0)f^i_\#(\gamma_1)\cdots f^i_\#(\gamma_k)\cdots$.
 The terms $\gamma_i$ are called the \textit{terms} of the splitting or \emph{splitting components} of $\gamma$.

  A \textit{CT map} or a \textbf{completely split relative train track map}
  is a  topological representative  with particularly nice properties. But CTs do not exist for all outer automorphisms. However, rotationless outer automorphisms are guaranteed to have a CT representative:

 \begin{lemma}\label{CT} \cite[Theorem 4.28]{FH-11}
  For each rotationless, hyperbolic $\phi\in \out$, there exists a CT map $f:G\rightarrow G$ such that $f$ is a relative train-track representative for $\phi$ and has the following properties:

 \begin{enumerate}
  \item \textbf{(Principal vertices)} Each principal vertex is fixed by $f$ and each periodic direction at a principal vertex is fixed by $T_f$. 
  Each vertex which has a link in two distinct irreducible strata is principal and a turn based at such a vertex with edges in the two distinct stratum is legal.
  \item \textbf{(Nielsen paths)} The endpoints of all indivisible Nielsen paths are principal vertices. 
  \item \textbf{(Zero strata)} Each zero stratum $H_i$ is contractible and there exists an EG stratum
  $H_s \text{ for some } s>i$ (see \cite[Definition 2.18]{FH-11}) such that each vertex of $H_i$ is contained in $H_s$ and the
  link of each vertex in $H_i$ is contained in $H_i \cup H_s$.
  \item \textbf{(Superlinear NEG stratum)} \cite[Lemma 4.21]{FH-11} Each non-fixed NEG stratum $H_i$ is a single oriented edge $E_i$ and has a splitting $f_\#(E_i) = E_i\cdotp u_i$, where $u_i$ is a nontrivial 
  circuit which is not a Nielsen path.
 \end{enumerate}
 \end{lemma}
 
 For any $\pi_1$-injective map $f: G_1\to G_2$ between graphs, there exists a constant $BCC(f)$, called the {\emph{bounded cancellation constant}} for $f$, such that for any lift $\til{f}: \til{G}_1\to \til{G}_2$ to the universal covers and any path $\til{\gamma}$ in $\til{G}_1$, the path $\til{f}_\#(\til{\gamma})$ is contained in a $BCC(f)$ neighbourhood of $\til{f}(\til{\gamma})$ (see \cite{Co-87} and   \cite[Lemma 3.1]{BFH-97}).

 \begin{definition}[label=critical]
   Let $f:G\to G$ be a CT map for  
  $\phi\in\out$, with $H_r$ being an exponentially growing stratum with associated Perron-Frobenius eigenvalue $\lambda$. 
  If $BCC(f)$ denotes the bounded cancellation constant for $f$, then the number $\frac{2BCC(f)}{\lambda-1}$ is called 
  the \emph{critical constant} for $H_r$. 
  
  \end{definition}
  It can be easily seen that for every number $C>0$ that exceeds the 
  critical constant, there is some $1\geq\mu>0$ such that if $\alpha\beta\gamma$ is a concatenation of $r-$legal paths where 
   $\beta $ is some $r-$legal segment of length $\geq C$, then the $r-$legal leaf segment of 
   $f^k_\#(\alpha\beta\gamma)$ corresponding to $\beta$ has length $\geq \mu\lambda^k{|\beta|}_{H_r}$ (see \cite[pp 219]{BFH-97}).
  To summarize, if we have a path in $G$ which has some $r-$legal ``central''  subsegment of length greater than the
    critical constant, then this segment is protected by the bounded cancellation lemma and under iteration, the
    length of this segment grows exponentially.


\noindent
\textbf{Nonattracting subgroup system:}
For any hyperbolic $\phi$, the \textit{non-attracting subgroup system} of an attracting lamination $\Lambda^+$ is a free factor system, denoted by $\mathcal{A}_{na}(\Lambda^+_\phi)$, and contains information about lines and circuits which are not attracted to the lamination.
We point the reader to  \cite{HM-20} for  the construction of the non-attracting subgraph whose fundamental group gives us this subgroup system \cite[Section 1.1]{HM-20}. We  list some key properties which we will be using. 
\begin{lemma}\cite[Theorem F, Corollary 1.7, Lemma 1.11]{HM-20}
\label{NAS}
 \begin{enumerate}
  \item A conjugacy class $[c]$ is not attracted to $\Lambda^+_\phi$ if and only if it is carried by $\mathcal{A}_{na}(\Lambda^+_\phi)$. No line carried by $\mathcal{A}_{na}(\Lambda^+_\phi)$ is attracted to $\Lambda^+_\phi$ under iterates of $\phi$. 
  \item $\mathcal{A}_{na}(\Lambda^+_\phi)$ is invariant under $\phi$ and does not depend on the choice of the CT map representing $\phi$. When $\phi$ is hyperbolic, $\mathcal{A}_{na}(\Lambda^+_\phi)$ is always a free factor system. 
  \item  Given $\phi, \phi^{-1} \in \out$ both rotationless, and a dual lamination pair $\Lambda^\pm_\phi$, we have $\mathcal{A}_{na}(\Lambda^+_\phi)= \mathcal{A}_{na}(\Lambda^-_\phi)$.
  \item If $\{\gamma_n\}_{n\in\mathbb{N}}$ is a sequence of lines or circuits such that every weak limit of every subsequence of $\{\gamma_n\}$ is carried by $\mathcal{A}_{na}(\Lambda^+_\phi)$ then $\{\gamma_n\}$ is carried by $\mathcal{A}_{na}(\Lambda^+_\phi)$ for all sufficiently large $n$.

 \end{enumerate}

\end{lemma}

\noindent
\textbf{Singular lines and nonattracted lines:}
\begin{definition}
A {\emph{singular line}} for $\phi$ is a line $\gamma\in \mathcal{B}$ such that there exists a principal lift $\Phi$ of some rotationless iterate of $\phi$
 and a lift $\widetilde{\gamma}$ of $\gamma$ such that the endpoints of $\widetilde{\gamma}$ are contained in Fix$_+(\widehat{\Phi}) \subset \partial \mathbb{F}$.
 \end{definition}

Recall (as per the discussion preceding Lemma \ref{rotationless}) that 
$\mathcal{B}_{\text{Fix}^+}(\phi)$
denotes the set of all singular lines of $\phi$.
 A {\emph{singular ray}} is a ray obtained by iterating a principal direction.

The following definition and the lemma after it is from  \cite{HM-20} and identifies the set of lines which do not get attracted to an element of $\mathcal{L}^+_\phi$.
\begin{definition}

Let $[A] \in \mathcal{A}_{na}(\Lambda^+_\phi)$ and $\Phi \in P(\phi)$, we say that $\Phi$ is $A-related$ if Fix$_+(\widehat{\Phi})\cap \partial A \neq \emptyset$. Define the extended boundary of $A$ to be $$\partial_{ext}(A,\phi) = \partial A \cup \bigg( \bigcup_{\Phi}Fix_+(\widehat{\Phi}) \bigg)$$
where the union is taken over all $A$-related $\Phi\in P(\phi)$.
\end{definition}
Let $\mathcal{B}_{ext}(A,\phi)$ denote the set of lines which have end points in $\partial_{ext}(A,\phi)$; this set is independent of the choice of $A$ in its conjugacy class. Define $$\mathcal{B}_{ext}(\Lambda^+_\phi)  = \bigcup_{[A]\in \mathcal{A}_{na}(\Lambda^+_\phi)} \mathcal{B}_{ext}(A,\phi)$$ 
For convenience we denote the collection of all generic leaves of all attracting laminations for $\phi$ by the set $\mathcal{B}_{gen}(\phi)$. 

\begin{lemma}\label{concat}
 \cite[Theorem 2.6]{HM-20}

 If $\phi, \psi=\phi^{-1} \in \out$ are rotationless and $\Lambda^+_\phi, \Lambda^-_\phi$ is a dual lamination pair, then the set of lines which are not attracted to $\Lambda^-_\phi$ are given by
 $$\mathcal{B}_{na}(\Lambda^-_\phi, \psi) = \mathcal{B}_{ext}(\Lambda^+_\phi) \cup \mathcal{B}_{gen}(\phi) \cup \mathcal{B}_{\text{Fix}^+}(\phi)$$

\end{lemma}

\noindent
\textbf{Structure of Singular lines:} The next Lemma, due to
Handel and Mosher, tells us the structure of singular lines and guarantees that one of the leaves of any attracting
lamination is a singular line.

\begin{lemma}\cite[Lemma 3.5, Lemma 3.6]{HM-09}, \cite[Lemma 1.63]{HM-20}\label{structure}
Let $\phi\in\out$ be rotationless and hyperbolic and let $l\in \mathcal{B}_{\text{Fix}^+}(\phi)$. Then:
\begin{enumerate}
 \item $l = \overline{R}\alpha R'$ for some singular rays $R \neq R'$ and some path $\alpha$ which is
 either trivial or a Nielsen path. Conversely, any such line is a singular line.
\item If $\Lambda\in\mathcal{L}^+_\phi$ then there exists a leaf of $\Lambda$ which is a singular line and one of its
ends is dense in $\Lambda$.

\end{enumerate}

\end{lemma}

\begin{lemma}\cite[Corollary 2.17, Theorem H]{HM-20}\cite[Theorem 6.0.1]{BFH-00} ({Weak attraction theorem}:)
\label{WAT}
 Let $\phi\in \out$ be  rotationless and exponentially growing. Let $\Lambda^\pm_\phi$ be a dual lamination pair for $\phi$. Then for any line $\gamma\in\mathcal{B}$ not carried by $\mathcal{A}_{na}(\Lambda^{+}_\phi)$ at least one of the following hold:
\begin{enumerate}
 \item $\gamma$ is attracted to $\Lambda^+_\phi$ under iterations of $\phi$.
   \item The weak closure of $\gamma$ contains $\Lambda^-_\phi$.
\end{enumerate}
Moreover, if $V^+_\phi$ and $V^-_\phi$ are attracting neighborhoods for the laminations $\Lambda^+_\phi$ and $\Lambda^-_\phi$ respectively, there exists an integer $M\geq0$ such that at least one of the following holds:
\begin{itemize}
 \item $\gamma\in V^-_\phi$.
\item $\phi_\#^m(\gamma)\in V^+_\phi$ for every $m\geq M$. 
\item $\gamma$ is carried by $\mathcal{A}_{na}(\Lambda^{+}_\phi)$.
\end{itemize}

\end{lemma}

For a hyperbolic outer automorphism, the following lemma shows that any conjugacy class is always weakly attracted to some element of $\mathcal{L}^+(\phi)$. By using Lemma \ref{NAS} we therefore know that every conjugacy class is also attracted to some element of $\mathcal{L}^-_\phi$
under $\phi^{-1}$. 
\begin{lemma}\cite[Proposition 2.21, Lemma 3.1, Lemma 3.2]{Gh-20}\label{EG}
Let $\phi\in\out$ be  rotationless and hyperbolic. Then:
\begin{enumerate}
    \item $\text{Fix}^+(\phi)$ is a finite set. 
    \item  Every conjugacy class is weakly attracted to some element of $\mathcal{L}^+_\phi$ under iterates of $\phi$.
    \item  The weak closure of every point in $\xi\in\mathsf{Fix}^+(\phi)$ contains an element  $\Lambda^+\in\mathcal{L}^+_\phi$.
\end{enumerate}
 
 \end{lemma}

Item (3) of the lemma characterizes the nature of its attracting fixed points. This is crucial to understanding the notion of ``independence of automorphisms" that we describe later.

\section{Legality, independence and stretching}\label{sec-lis}
We begin this section by describing a notion of \emph{legality of paths} which we will use in our proof. Multiple versions of such
a notion exist,  all adapted to gaining quantitative
control over  exponential growth. Let $\phi \in \out$ be  hyperbolic and rotationless. Let $f: G\to G$ denote the $CT$ map representing $\phi$. 
 Let ${|\alpha|}_{H_r}$  denote the $r-$length 
 of a path $\alpha$ in $G$, i.e.\  we only count the edges of $\alpha$ contained in $H_r$. 

\subsection{Legality and Attraction of lines}\label{sec-li}

Recall the definition of critical constant (after Lemma \ref{CT}) for an exponentially growing stratum and 
the legality ratio of paths in \cite[Definition 3.3]{Gh-18}. This notion of legality ratio was first introduced in \cite[pp-236]{BFH-97} for fully irreducible hyperbolic elements. In the fully-irreducible setting there is only one stratum, and it is exponentially growing. So the notion is a lot simpler. For our use we adapt the definition to make it work for reducible hyperbolic elements. \\

\noindent {\bf Legality ratio of paths:}
For a path $\alpha$ with endpoints at vertices
	of an exponentially growing stratum $H_r$ and  entirely contained in the union of $H_r$ and  a zero stratum which shares vertices with $H_r$ (see item (3) of Lemma \ref{CT}),   decompose $\alpha$ into a concatenation of paths each of which is either a path in $G_{r-1}$ or a path of height $r$. We consider
	components $\alpha_i$
	(if such exist) in this decomposition of  $\alpha$ such that
	\begin{enumerate}
	\item $\alpha_i$ is of 
	height $r$ and is a segment of a generic leaf.
	\item ${|\alpha_i|}_{H_r} \geq C$, where $C$ is the critical constant for $H_r$.
	\end{enumerate}
    Next, consider
     the ratio   $$\frac{ \sum{|\alpha_i|}_{H_r}}{|\alpha|}$$
 for such a decomposition. 
  The \emph{ $H_{r}$-legality}  of $\alpha$ is defined as the maximum of the above ratio over all such decompositions of $\alpha$  and is denoted by 
  ${LEG}_r(\alpha)$. The maximum is realised for some decomposition of  $\alpha$. 
  For such a decomposition,  denote by $\alpha_k'$ ($1\leq k\leq n$) the subpaths which contribute to the $H_r$-legality of  $\alpha$. Set $L(\alpha) = \sum_k |\alpha_k'|_{H_r}$.
 
If $\beta$ is any finite edge-path in $H_r$, we use Lemma \ref{CT}  to get a splitting $\beta = \beta_1 \cdot \beta_2 \cdots \beta_k$, where each $\beta_i$ is either a path entirely contained in an irreducible stratum or a maximal path contained in the union of an exponentially growing stratum and
 	a zero stratum as in item(3) of Lemma \ref{CT}. We define $$LEG(\beta) = \bigg(\sum_{s_i} L(\beta_{s_i})\bigg) / |\beta|,$$  where $\beta_{s_i}$ is one of the  components in the decomposition of $\beta$ of height $s_i$, and $H_{s_i}$ is exponentially growing.  Components which do not cross an exponentially growing stratum are ignored in this sum.\\

The following proposition says that a circuit with not too many illegal turns gains legality   under iteration. If $\phi$ is fully
irreducible, the proof can be found in \cite[Lemma 5.6]{BFH-97}. We adapt the idea of that proof to our definition of legality. To see how this proof reduces to the fully irreducible case, recall that for a fully irreducible hyperbolic automorphism the non-attracting subgroup system in trivial. Therefore the weak attraction theorem Lemma \ref{WAT} reduces to the statement that any line whose closure does not contain the repelling lamination necessarily converges to the attracting lamination under iteration of $\phi$. So the limiting line $\ell$ in the proof below has all desired properties on the nose. 

\begin{prop}[{Legality}] \label{legality}
Let $\phi\in \out$  be hyperbolic and rotationless and $f: G\to G$ be a CT map representing $\phi$. Let $C$ be some number greater than all the critical constants associated to  exponentially growing strata. Let $V^+, V^-$ denote the union of attracting neighbourhoods for elements of $\mathcal{L}^+_\phi, \mathcal{L}^-_\phi$
respectively, where the leaf segments defining these neighbourhoods have length $\geq 2C$ and $V^+$ does not contain any leaf of any element of $\mathcal{L}^-_\phi$ and $V^-$ does not contain any leaf of any element of $\mathcal{L}^+_\phi$.

Then there exists some $\epsilon >0, N_0>0$ such that for every circuit  $\beta$ in $G$ with the property that  $\beta\in V^+$, $\beta\notin V^-$  we have $LEG(f^n_\#(\beta)) \geq \epsilon$ for all $n \geq N_0$. 
\end{prop}

\begin{proof} We argue by contradiction.
 Suppose the conclusion is false. Then there exists a sequence $n_j \to \infty$ and circuits $\alpha_j$ satisfying the hypothesis such that $LEG(f^{n_j}_\#(\alpha_j)) \to 0$.  Since $\alpha_j \in V^+$ we have that $LEG(\alpha_j) \neq 0$ for every $j$. Therefore we may assume $|\alpha_j| \to \infty$. Now we choose subpaths $\delta_j$ of $\alpha_j$ such that the following hold: 
 \begin{enumerate}
     \item $\delta_j \notin V^-$ and $|\delta_j| \to \infty$. 
     \item  $LEG(f^{n_j}_\#(\delta_j)) = 0$. 

 \end{enumerate}
 To see why item (2) holds, observe that since $LEG(f^{n_j}_\#(\alpha_j)) \to 0$,  $\alpha_j$'s do not contain sufficiently many long subpaths which are generic leaf segments of elements of $\mathcal{L}^+_\phi$ for the legality to grow under iterates of $f_\#$.  Therefore as $j\to \infty$ subpaths of $\alpha_j$ which are not generic leaf segments become arbitrarily large since $|\alpha_j|\to \infty$.
 
 Since $|\delta_j| \to \infty$ we may assume that $\delta_j$ is a circuit for all sufficiently large $j$. 
 Item (2) implies that $\delta_j \notin V^+$, since $f_\#(V^+) \subset V^+$. Since there are only finitely many elements in $\mathcal{L}^+_\phi$, applying item (2) of Lemma \ref{EG} we may pass to a subsequence if necessary, and assume that $\delta_j$'s are not carried by the non-attracting subgroup system corresponding to some fixed attracting lamination  $\Lambda^+\in\mathcal{L}^+_\phi$.  
 
 Therefore by item (4) of Lemma \ref{NAS} there exists a weak limit $\ell$ of the $\delta_j$'s  such that  $\ell$ is not carried $\mathcal{A}_{na}(\Lambda^+_\phi)$. Also note that  our assumption that $\delta_j \notin V^-$  implies that $\ell \notin V^-$, since $V^-$ is an open set. This implies that $\ell$ is not in the attracting neighbourhood of the dual lamination $\Lambda^-$ of $\Lambda^+$, which is contained in $V^-$.   Lemma \ref{WAT} applied to the dual lamination pair  $\Lambda^+, \Lambda^-$   then implies that $f^{n_j}_\#(\ell) \in V^+$ for all $j$ sufficiently large. Since $V^+$ is an open set, there exists some $J>0$ such that $f^{n_j}_\#(\delta_j)\in V^+$ for all $j\geq J$.  This  violates item (2) above. 
\end{proof}

The following result is a generalisation of \cite[Lemma 5.5, item (1)]{BFH-97} and is a direct consequence of the above proposition and the definition of critical constant.

\begin{lemma}[{Exponential growth}] \label{expgrowth}
Let $\phi\in \out$  be hyperbolic and rotationless and $f: G\to G$ be a CT map representing $\phi$. Let $C$ be some number greater than the critical constants associated to all exponentially growing strata. Suppose $V^+, V^-$ denote the union of attracting and repelling neighbourhoods for $\phi$, where the leaf segments defining these neighbourhoods have length $\geq 2C$ and and $V^+$ does not contain any leaf of any element of $\mathcal{L}^-_\phi$ and $V^-$ does not contain any leaf of any element of $\mathcal{L}^+_\phi$.

Then  for every $ A>0$, there exists $N_1 > 0$ such that for every circuit  $\beta$ in $G$ with the property that  $\beta\in V^+$, $\beta\notin V^-$ we  have $|f^n_\#(\beta)| \geq A |\beta|$ for all $n \geq N_1$. 
\end{lemma}

\begin{proof}
 By Proposition \ref{legality}, there exists  $N_0$ such that for any circuit $\beta$ satisfying the hypothesis we have $LEG(f^n_\#(\beta)) \geq \epsilon$ for all $n \geq N_0$.  Let $\alpha = f^{N_0}_\# (\beta)$. By taking a splitting of  $\alpha$ as in the
 definition of legality, we obtain $\sum \{L(\alpha_i)\} \geq \epsilon |\alpha|$. If $\lambda$ is the minimum of the stretch factors corresponding to the exponentially growing strata of $f$, we get $$|f^k_\#(\alpha)| \geq D \lambda^k \sum_i  \{L(\alpha_i)\} \geq D\lambda^k \epsilon |\alpha|$$ for some constant $0< D\leq 1$ arising out of  the critical constant (see the role of $\mu$ in discussion after Definition \ref{critical}).  Since $N_0$ is fixed, we may choose $N_1$ large enough, independent of $\beta$ (due to the bounded cancellation property),  such that $D\lambda^{N_1} \epsilon |\alpha| \geq A |\beta|$. The result then follows for all $n \geq N_1$. 
 \end{proof}

 Following \cite{Gh-20}, we write $\mathcal{WL}^+(\phi) = \mathcal{B}_{gen}(\phi) \cup \mathcal{B}_{\text{ Fix}^+}(\phi)$ for any hyperbolic outer automorphism $\phi$. Recall that $\mathcal{B}_{gen}(\phi)$ is the set of all generic leaves of attracting laminations for $\phi$ and $\mathcal{B}_{\text{ Fix}^+}(\phi)$ denotes set of all singular lines. Similarly replacing $\phi$ by $\phi^{-1}$ we get $\mathcal{WL}^-(\phi)$. Set $\til{\mathcal{WL}}^+(\phi)$ to be the preimage  of $\mathcal{WL}^+(\phi)$ in $\til{\mathcal{B}}$. Similarly define $\til{\mathcal{WL}}^-(\phi)$.  The following lemma identifies lines which are weakly attracted to some element of $\mathcal{L}^+_\phi$ under iteration by $\phi$.
 
 Suppose  $\phi\in\out$ is fully-irreducible, rotationless and hyperbolic. Since there is only one attracting lamination and its non-attracting subgroup system is trivial, using Lemma \ref{concat} we get that $\ell$ is weakly attracted to $\Lambda^+$ if and only if $\ell \notin \mathcal{WL}^-(\phi)$. We want to extend this observation to the reducible case too. However the reducible hyperbolic case requires some more work and the statement needs some  modification primarily due to the possibility of existence of non-generic leaves of attracting laminations in the reducible case. 

\begin{lemma}[Attraction of lines]\label{aol}
Let $\phi \in \out$ be rotationless and hyperbolic and $f: G \to G$  be a completely split train-track map representing $\phi$. If 
$\til{\ell} \in \til{\mathcal{B}}$ is  such that $\til{\ell}$ is not asymptotic to any element of $\widetilde{\mathcal{WL}}^-(\phi)$, then $\ell$ is weakly attracted to some element of $\mathcal{L}^+_\phi$ under iterates of $\phi$.

\end{lemma}

\begin{proof}
 Suppose $\ell$ is not attracted to any element of $\mathcal{L}^+_\phi$. Then by the structure of non-attracted lines in Lemma \ref{concat}, we get that $\ell$ must be carried by the non-attracting subgroup system of every element of $\mathcal{L}^+_\phi$. If one of the non-attracting subgroup systems is trivial, then this immediately gives us a contradiction. Therefore we assume that none of them are trivial. By using the minimality of the free factor support $[K]$ of $\ell$ and the fact that every non-attracting subgroup system is a free factor system (item (2) of Lemma\ref{NAS}), we see that $[K]$ is carried by the non-attracting subgroup system of every element of $\mathcal{L}^+_\phi$.
 If $\sigma$ is any conjugacy class in $[K]$, then it cannot get attracted to any element of $\mathcal{L}^+_\phi$ under iterates of $\phi$, by  item (1) of Lemma \ref{NAS}. This contradicts conclusion (2) of Lemma \ref{EG}. 
 \end{proof}

\subsection{Independence and Stretching}\label{sec-stretch}
We fix  a homogeneous graph of roses $\pi: \XX \to \GG$ for the rest of the paper (cf.\ Definition \ref{def-tree} and the subsequent discussion). The universal cover is a  homogeneous tree of trees $\Pi: \YY \to \TT$ 
 The vertex set $V(\GG)$ (resp.\ 
edge set $E(\GG)$) of $\GG$ is denoted  by $\VV$ (resp. $\EE$).  The marked rose over $v\in \VV$ (resp. $e\in \EE$) is denoted
as $R_v$ (resp. $R_e$). Equip each $R_v$ (resp. $R_e$) with a base point $b_v$ (resp. $b_e$). Similarly, 
the marked tree over $v\in V(\TT)$ (resp. $e\in E(\TT)$) is denoted
as $T_v$ (resp. $T_e$). Base-points in $T_v$ (resp. $T_e$) are denoted by $\tilde b_v$ (resp. $\tilde b_e$).

We associate with each oriented edge $e$, a tuple  $(G_e,\Phi_e, f_e,{{q}}_{ev}, \rho_e)$ given by the following data:
\begin{enumerate}
      \item Let $e=[v,w]$ be an edge. Then $G_e$ is a marked graph with marking induced by  $R_e$. Under the edge-to-vertex map, $\pi_1(G_e, b_e)$ maps injectively  to a finite index subgroup of $\pi_1(R_v, b_v)$. 
     \item  $\Phi_e$ is an automorphism of $\pi_1(G_e, b_e)$. 
    \item  $f_e$ is a completely split train-track map on $G_e$ representing an  outer automorphism in the outer automorphism class of some rotationless power of $\Phi_e$ (see Lemma \ref{rotationless}) \item The lift  of $f_e$ to the universal cover is given by $\til{f}_e:(\til{G}_e,\tilde b_e) \to (\til{G}_e,\tilde b_e)  $.  
   \item The lift  to the universal cover of the  map from $G_e$ to $R_v$ 
   is  given by ${q}_{ev} : \til{G}_e \to \til{R}_v$. Note that ${q}_{ev}$ is a  quasi-isometry with uniform constants.
\end{enumerate}
Let $E$ denote the edge $e$ with reverse orientation. We have a base-point preserving change of markings map $\rho_e: G_E \to G_e$ and  its lift $\til{\rho}_e:\til{G}_E \to \til{G}_e $  to  universal covers. 

We fix the following notation for the purposes of this subsection.

\begin{enumerate}[label=(\alph*)]
\item Let $v\in \GG$ be any vertex and let $e_1, \cdots, e_n$ be all the edges of $\GG$  originating at $v$. We will use  $G_i, f_i, {q}_{iv}$ to denote $G_{e_i}, f_{e_i}, {q}_{e_iv}$ respectively.

\item The set of all attracting and repelling laminations of $\phi_i$ will be denoted by $\mathcal{L}^+_i$ and $\mathcal{L}^-_i$ respectively. $\mathcal{L}^\pm_i := \mathcal{L}^+_i\cup\mathcal{L}^-_i$.

\item $\til{\mathcal{B}}_i$ denotes the space $\{\partial \til G_i \times \partial \til  G_i - \Delta \}/ \mathbb{Z}_2$ and $\mathcal{B}_i$ denotes its image under the quotient by $\pi_1(G_i)$. $\til{\mathcal{B}}_v$ and $\mathcal{B}_v$ are defined similarly using $\pi_1(R_v)$. The quotient spaces are equipped with the weak topology.

\item $\widehat{q}_{iv}: \partial\til{G}_i \to \partial\til{R}_v$ denotes the homeomorphism between boundaries induced by $q_{iv}$. We use $\widehat{q}_{vi}$ to denote the inverse homeomorphism. $\widehat{q}_{iv} \times \widehat{q}_{iv}$ extends to a homeomorphism of the corresponding product spaces which induces a homeomorphism of the spaces $\til{\mathcal{B}}_i$ and $\til{\mathcal{B}}_v$. We will abuse the notation and continue to denote this induced homeomorphism by $\widehat{q}_{iv} \times \widehat{q}_{iv}$. Use $\widehat{q}_{vi} \times \widehat{q}_{vi}$ to denote the corresponding inverse homeomorphism.

\item If $\gamma_i \in \til{\mathcal{B}}_i$, then $\gamma_i^v$ denotes the image $\widehat{q}_{iv} \times \widehat{q}_{iv} (\gamma_{i})$. We will call $\gamma_i$  the realisation of $\gamma^v_i$ in $\til{G}_i$. 
If $X$ is a subset of $\til{\mathcal{B}}_i$ then $X^v$ denotes the union of $\gamma^v_i$'s as $\gamma_i$ ranges over all elements of $X$.

\item  $ \mathcal{B}_{gen}(\phi) \cup \mathcal{B}_{\text{ Fix}^+}(\phi) = \mathcal{WL}(\phi)$ is   closed and $\phi$-invariant (\cite[Theorem 3.10]{Gh-18}) for any hyperbolic outer automorphism $\phi$. We use the notation $\mathcal{WL}^+_i, \mathcal{WL}^-_i$ to denote the  set of lines $\mathcal{WL}(\phi_i), \mathcal{WL}(\phi_i^{-1})$ respectively. Also, let $\mathcal{WL}^\pm_i=\mathcal{WL}^+_i \cup \mathcal{WL}^-_i$.

\end{enumerate}
We shall refer to the Notation in (1)-(5) above along with (a)-(f) as the \emph{standard setup} for the rest of this section. The following definition is a modification of the corresponding definition
of independence of surface automorphisms from \cite{min}.
\begin{definition}\label{ind}
{\emph{(Independence of automorphisms:)}}  Let $H_1, H_2$ be finite index subgroups of a 
free group $F$ with indices $k_1, k_2$ respectively.
Let  $\Phi_1, \Phi_2$ be hyperbolic automorphisms of $H_1, H_2$ respectively. Let ${\{a_i\cdot H_1\}}_{i=1}^{k_1}$ and ${\{b_j\cdot H_2\}}^{k_2}_{j=1}$  be the collections of distinct cosets of $H_1, H_2$ in $F$. We will say  that $\Phi_1, \Phi_2$ are \textit{independent in } $F$  if the following conditions are satisfied:

\begin{enumerate}[label=(\Alph*)]
 \item $a_i\cdot(\til{\ell}^v_1) \text{ and } a_j\cdot(\til{\ell}^v_2)$ do not have a common end in $\partial F$ for any $\til{\ell}_1, \til{\ell}_2\in \til{\mathcal{WL}}^\pm_1$ 
 where $1\leq i\neq j\leq k_1$.  Similarly, $b_i\cdot(\til{\ell}^v_1) \text{ and } b_j\cdot(\til{\ell}^v_2)$ do not have a common end in $\partial F$ for any $\til{\ell}_2, \til{\ell}_2\in \til{\mathcal{WL}}^\pm_2$
  $1\leq i\neq j\leq k_2$.
 
 \item $a_i\cdot(\til{\ell}^v_1) \text{ and } b_j\cdot(\til{\ell}^v_2)$ do not have a common end in $\partial F$   for any $\ell_i\in\til{\mathcal{WL}}^\pm_i$ for all $1\leq i \leq k_1, 1\leq j \leq k_2$.
\end{enumerate}

\end{definition}

As an immediate consequence of the fact that $\widehat{q}_{1v}\times\widehat{q}_{1v}: \til{\mathcal{B}}_1 \to \til{\mathcal{B}}_v$ is a homeomorphism, we have the following.

\begin{lemma}[Disjointness is preserved]
\label{dip}
If $\til{\ell}^v\in\til{\mathcal{B}}_v$ is such that $\til{\ell}^v$ is not asymptotic to any element of $\bigcup\limits_{s=1}^{k_1}a_s\cdot \til{\mathcal{WL}}^{\pm v}_1$, then  the realisation of $\til{\ell}^v$ in $\til{G}_1$ is not asymptotic to any lift of any element of $\mathcal{WL}^\pm_1$.

\end{lemma}

Given the standard setup for this section, let $v$ be a vertex of $\GG$ and let $e_1, e_2, \cdots, e_n$ be all the oriented edges in $\GG$ which have $v$ as the initial vertex. We Will say that the automorphisms $\Phi_1, \Phi_2, \cdots, \Phi_n$ associated to these edges are \emph{independent} in $\pi_1(R_v)$ if $\Phi_i, \Phi_j$ are independent in $\pi_1(R_v)$ for any $1\leq i\neq j \leq n$.
\begin{lemma}[Independence implies attraction] \label{iia} Given the standard setup for this section, let $v$ be a vertex of $\GG$ and let $e_1, e_2, \cdots, e_n$ be all the oriented edges in $\GG$ which have $v$ as the initial vertex. If the automorphisms $\Phi_1, \Phi_2, \cdots, \Phi_n$ associated to these edges are independent in $\pi_1(R_v)$ then for all $i\neq 1$ the projection  to $G_1$ of the image of any lift of any leaf of any attracting or repelling laminations of $\phi_i$  is weakly attracted to  some element of $\mathcal{L}^+_1$  under iterates of $\phi_1$. (An analogous statement holds for $\mathcal{L}^-_1$ and $\phi^{-1}_1$). 

\end{lemma}

\begin{proof}
    Every leaf of an attracting lamination for $\phi_i$ is an element of $\mathcal{WL}^+_i$ (see  \cite[Corollary 3.8]{Gh-20}). Since $\Phi_i, \Phi_1$ are independent in $\pi_1(R_v)$, it
    follows from Definition \ref{ind}  that translates of elements of $\til{\mathcal{WL}}^{\pm v}_i$ are not asymptotic to translates of elements of $\til{\mathcal{WL}}^{\pm v}_1$, for  $i\neq 1$. 
    
    Using  Lemma \ref{dip}, we see that the image (under the homeomorphism between $\til{\mathcal{B}}_i$ and $\til{\mathcal{B}}_1$) in $\til{\mathcal{B}}_1$ of the lift of any leaf of any attracting or repelling lamination of $\phi_i$ is not asymptotic to an element of $\widetilde{\mathcal{WL}}^{\pm}_1$. By using Lemma \ref{aol} we see that 
    its projection to $G_1$ gets weakly attracted to some element of $\mathcal{L}^+_1$ under iterates of $\phi_1$. 
    
    A similar argument gives us the result for $\mathcal{L}^-_1$. 
\end{proof}

\begin{rmk}
	
 The proof of this lemma is  easier when both $\phi_1, \phi_2$ are fully irreducible. In that case,  a line is attracted  to the unique attracting lamination for $\phi_1$ if and only if it is not in $\mathcal{WL}^-_1$. But the projection to $G_1$ of the image of any lift  of any leaf of $\Lambda^\pm_2$ cannot be in $\mathcal{WL}^-_1$ as a consequence of the definition of independence. 
\end{rmk}

The following Lemma upgrades the disjointness conditions of Definition \ref{ind} to disjointness
of neighborhoods.
\begin{lemma}[Disjoint neighbourhoods exist]\label{dne}
 Given the standard setup for this section, let $v$ be a vertex of $\GG$ and let $e_1, e_2$ be two oriented edges in $\GG$ which have $v$ as the initial vertex. Let the automorphisms $\Phi_1, \Phi_2$ associated to these edges be independent in $\pi_1(R_v)$. Let the index in $\pi_1(R_v)$ of the group associated to edge $e_i$ be $k_i$. 
 Then for $\epsilon_1, \epsilon_2 = +,- $, there exist open sets $V^{\epsilon_i}_i\subset \mathcal{B}_i$ such that 
 \begin{enumerate}[label=(\roman*)]
  \item Every attracting lamination of $\phi_i$ is contained in $V^+_i$ and every repelling lamination of $\phi_i$ is contained in $V^-_i$. Also, $V^+_i \cap V^-_i = \emptyset$ for $i= 1,2$.
  \item The projection to $G_1$ of the image (using the homeomorphism between $\til{\mathcal{B}}_2$ and $\til{\mathcal{B}}_1$) of any lift of a generic leaf of any attracting or repelling lamination of $\phi_2$ is not contained in $V^+_1\cup V^-_1$. A similar condition holds with roles of $\phi_1, \phi_2$ interchanged. 
  \item $a_i\cdot  \widehat{q}_{1v}\times \widehat{q}_{1v}  (\til{V}^{\epsilon_1}_1) \cap a_j\cdot   \widehat{q}_{1v}\times \widehat{q}_{1v}(\til{V}^{\epsilon_2}_2) = \emptyset$ 
 where $1\leq i\neq j\leq k_1$.  Analogous result for $\til{V}^+_2$ and $\til{V}^-_2$.
 
  \item For any lift $\til{V}^{\epsilon_i}_i \subset \til{\mathcal{B}}_i$, we have $a_s\cdot(\widehat{q}_{1v}\times \widehat{q}_{1v}(\til{V}^{\epsilon_1}_1)) \cap b_t\cdot(\widehat{q}_{2v}\times \widehat{q}_{2v}(\til{V}^{\epsilon_2}_2)) = \emptyset$ for every $1\leq s\leq k_1, 1\leq t \leq k_2$. 
 
 \end{enumerate}

\end{lemma}

\begin{proof}
  For every attracting lamination $\Lambda^+\in\mathcal{L}^+_i$, pick a generic leaf of $\Lambda^+$ and choose an attracting neighbourhood of $\Lambda^+$ defined by a finite segment  of the generic leaf. Denote the union (over the finitely many attracting laminations of $\phi_i$) of  such attracting neighbourhoods by $V^+_i$. Do the same with $\phi^{-1}_i$ to construct $V^-_i$ for $i=1,2$. By choosing the segments long enough  conclusion (i) can be satisfied.
   
  By using condition $(B)$ of definition \ref{ind}  and \ref{dip} the projection of the image of any lift  of any leaf of $\Lambda^\pm_j\in\mathcal{L}^\pm_j$ in $G_i$ does not have a common end with a generic leaf of any attracting or repelling lamination of $\phi_i$, for $1\leq i\neq j\leq 2$. By using the birecurrence property of a generic leaf we may take longer generic leaf segments and replace $V^+_1$ with a smaller open set such that the projection in $G_1$ of the image of any lift of any generic leaf of any attracting or repelling lamination of $\phi_2$ is not in $V^+_1$. Similarly construct $V^-_1$. Interchanging the role of $\phi_1$ and $\phi_2$, we construct $V^+_2, V^-_2$. Hence conclusion (ii) is also satisfied.

 To show that (iii) holds we use the first condition in the definition of independence. 
 Having constructed neighbourhoods which satisfy conditions (i) and (ii), suppose that (iii) is violated for all such open sets satisfying (i) and (ii). For concreteness assume that $\til{\ell}_{n} \in \widehat{q}_{1v}\times \widehat{q}_{1v} (\til{V}^+_{1n}) \cap a\cdot\widehat{q}_{1v}\times \widehat{q}_{1v}(\til{V}^+_{1n})$ for some $a_i = a$ and $V^+_{1n}$ are a sequence of nested  open neighbourhoods constructed by choosing longer and longer generic leaf segments. If the limit of the sequence $\til{\ell}_n$ is $\til{\ell}$, then $\til{\ell} \in \til{\mathcal{WL}}^{+v}_1 \cap a\cdot\til{\mathcal{WL}}^{+v}_1$, which violates condition $(A)$ of the independence criterion. This proves (iii).
 
 Next, suppose that property  (iv) is violated for every choice of open sets satisfying (i), (ii), (iii).
 Then there exists a sequence of integers $n_k\to \infty$ and corresponding open sets $V^+_{1n_k}, V^+_{2n_k}$ ( and $V^-_{1n_k}, V^-_{2n_k}$) which are a union of attracting (and repelling ) neighbourhoods  defined by generic leaf segments of length greater than $n_k$, such that conclusion (iv) is violated. We may further choose the finite segments defining the attracted neighbourhoods so that the sequence of open sets $V^+_{n_k}$ is nested and decreasing (with respect to inclusion). 
 
 Since we have only finitely many $a_s, b_t$, after passing to a subsequence we may assume that condition (iv) is violated for a fixed $s \text{ and } t$ for the open sets constructed above. After passing to a further subsequence we may assume for sake of concreteness that $a_s\cdot \widehat{q}_{1v}\times\widehat{q}_{1v}(\til{U}^+_{1n_k}) \cap b_t\cdot \widehat{q}_{2v}\times \widehat{q}_{2v}(\til{U}^+_{2n_k}) \neq \emptyset$ for all sufficiently large $k$, where $U^+_{1n_k}$ is a nested sequence of open sets in $\mathcal{B}_1$ which are defined by choosing an increasing sequence of generic leaf segments of some fixed element of $\mathcal{L}^+_1$. A similar assumption can, by the same reasoning, be made for $U^+_{2n_k}$.

  Note that as $k\to \infty$ the intersection of all the open sets $a_s\cdot(\widehat{q}_{1v}\times \widehat{q}_{1v}(\til{U}^+_{1n_k}))$ is nonempty and equals $a_s\cdot \widehat{q}_{1v}\times \widehat{q}_{1v}(\til{\mathcal{WL}}^+_1)$. A symmetric conclusion holds for $U^+_{2n_k}$.
  Since both $a_s\cdot \widehat{q}_{1v}\times \widehat{q}_{1v}(\til{\mathcal{WL}}^+_1)$ and $b_t\cdot \widehat{q}_{2v}\times \widehat{q}_{2v}(\til{\mathcal{WL}}^+_2)$ are closed sets, this implies that there exists some element $\til{\ell}\in\mathcal{B}_v$ at least one of whose endpoints in $\partial\til{R}_v$ lies in both $a_s\cdot \widehat{q}_{1v}\times \widehat{q}_{1v}(\til{\mathcal{WL}}^+_1)$ and $b_t\cdot \widehat{q}_{2v}\times \widehat{q}_{2v}(\til{\mathcal{WL}}^+_2)$. This contradicts independence of the automorphisms.
 
\end{proof}

We are now ready to prove our version of the
3-out-of-4 stretch lemma (see \cite{mosher-hbh,min}) which establishes the hallway flaring 
condition (Definition \ref{defnofflare}) for us. For ease of notation we will use ${f}^+_i: G_i \to G_i$ to denote the $CT$ map for the outer automorphism $\phi_i$ associated to the edge $e_i$ and ${f}^{-}_i: G^-_i \to G^-_i$ to denote the $CT$ map associated to the inverse outer automorphism $\phi^{-1}_i$. For a finite geodesic path $\til{\tau} \in \til{R}_v$ we say that $\til{\tau}_i$ is its realisation in $\til{G}_i$ if $\til{\tau}_i$ is a geodesic edge-path in $\til{G}_i$ joining the images of the end-points of 
$\til{\tau}$ under the quasi-isometry from $\til{R}_v$ to  $\til{G}_i$. Also, for ease of notation, we will just write $|f^m_{i\#}(\alpha)|$ where it is understood that this length is being measured on the marked graph on which $f_i$ is defined. 
The same convention will be used for lifts to universal covers. By $|G_i|$ we denote the number of edges in $G_i$, and similarly    $|G^-_i|$ to denote the number of edges in $G_i^-$.

\begin{prop}[3-out-of-4 stretch]\label{3of4}
 Given the standard setup, let $v$ be a vertex of $\GG$ and let $e_1, e_2$ be two oriented edges in $\GG$ which have $v$ as the initial vertex. If the automorphisms $\Phi_1, \Phi_2$ associated to these edges are independent in $\pi_1(R_v)$, then there exists some constants $M'_v , L'_v >0$ such that  for every geodesic segment $\til{\tau}$  in $\til{R}_v$ of length greater than $L'_v$, we will get at least three  of the four numbers $|\til{f}_{i\#}^{\pm m}(\til{\tau}_i)|$ to be greater than $2 |\til{\tau}| $ for every $m > M'_v$, where $\til{\tau}_i$ is a realisation of $\til{\tau}$ in $\til{G}_i$ and $i = 1, 2$.
\end{prop}

\begin{proof}
 
 Let $A$ denote a number greater than twice the bounded cancellation constants for the $CT$ maps $f_i^\pm$ for ($i = 1,2$) and the quasi-isometry constants for the maps $q_{iv}$ and their inverses. Also assume that $A$ is greater than twice the bounded cancellation constants for the finitely many marking maps and change of marking maps involved  and their lifts to the universal covers.  By increasing $A$ if necessary assume that it is greater than the critical constants associated to each exponentially growing stratum of $f^+_i$ and $f^-_i$.

For every attracting lamination $\Lambda^+\in\mathcal{L}^+_i$, pick a generic leaf of $\Lambda^+$ and choose an attracting neighbourhood of $\Lambda^+$ defined by a finite segment  of the generic leaf of length greater than maximum of $\{2A, 2|G_i|, 2|G^-_i|\}$. By taking longer generic leaf segments if necessary, assume that we have open sets of $\mathcal{B}_i$ (for $i=1, 2$) which satisfy the conclusions of Lemma \ref{dne}.

By Lemma \ref{aol} we know that any line in $G_i$ which does not have a lift that is asymptotic to an element of $\til{\mathcal{WL}}^\pm_i$ is weakly attracted to some element of $\mathcal{L}^+_i$.  By applying Lemma \ref{WAT} to each dual lamination pair of $\phi_i$ and taking the maximum over all exponents, we obtain some integer $m_i$
 such that $f_{i\#}^{m_i} (\ell) \in V^+_i$ for any line  $\ell\notin V^-_i$ where $i=1, 2$. 
We do the same for the inverses of $\phi_1, \phi_2$  and get constants $m_i'$. Let $M'_0 > $ maximum of  $\{m_1, m_2, m_1', m_2'\}$.

We claim that there exist constants $M'_v > M'_0,\, L_v >0$ such that  for every geodesic segment $\til{\tau}$  in $\til{R}_v$ of length greater than $L_v$, we will get at least 3 of the 4 numbers $|\til{f}_{i\#}^{\pm m}(\til{\tau}_i)|$ to be greater than $2 |\til{\tau}| $ for all $m> M'_v$. 

We argue by contradiction.
Suppose not. Then there exists a sequence of positive integers $n_j \to \infty$ and paths $\til{\sigma}_j \in \til{R}_v$ with $|\til{\sigma}_j| > j$, such that at least two of the numbers $|\til{f}_{i\#}^{n_j}(\til{\sigma}_{ij})|$ is less than $2 |\til{\sigma}_j|$ as $i$ varies. Since $\Phi_1, \Phi_2$  are both hyperbolic, the associated mapping tori are hyperbolic
\cite{BF,brinkmann}.  Hence the hallways flare condition (Definition \ref{defnofflare}) holds \cite[Section 5.3]{mahan-sardar}. So we may pass to a subsequence and assume without loss of generality  that $|\til{f}_{1\#}^{\pm n_j}(\til{\sigma}_{1j})|, |\til{f}_{2\#}^\pm {n_j}(\til{\sigma}_{2j})| < 2|\til{\sigma}_j|$ for all $j$. By the uniform bound on quasi-isometry constants,  we can write $|\til{\sigma}_j| \leq B|\til{\sigma}_{ij}| + 2K$ for $i = 1,2$ and some uniform constants $B, K>0$. The inequalities then transform to 
$$ \frac{|\til{f}_{1\#}^{n_j}(\til{\sigma}_{1j})|}{|\til{\sigma}_{1j}|}, \frac{|\til{f}_{2\#}^{n_j}(\til{\sigma}_{2j})|}{|\til{\sigma}_{2j}|} < \til{C} $$ for some uniform constant $\til{C}$. 
 Let $\sigma_{ij}$ denote the projection of $\til{\sigma}_{ij}$ to $G_i$.  We then get
 \begin{equation}\label{eq}
 \frac{|{f_1}_\#^{n_j}(\sigma_{1j})|}{|\sigma_{1j}|}, \frac{|{f_2}_\#^{n_j}(\sigma_{2j})|}{|\sigma_{2j}|} < C
 \end{equation}   for some uniform constant $C$ . 
Without loss of generality,
assume that $\til{\sigma}_j$ are all based at some fixed vertex in $\til{R}_v$, corresponding to the identity element of $\pi_1(R_v)$. By passing to a limit we get a geodesic line $\til{\ell}$ in $\til{R}_v$ with distinct endpoints in $\partial\til{R}_v$. By using item (iv) of Lemma \ref{dne} we get that 
$\til{\ell}$ cannot belong to both $\bigcup\limits_{s=1}^{k_1}\{a_s\cdot \widehat{q}_{1v}\times\widehat{q}_{1v}(\til{V}^-_1)\}$ and $\bigcup\limits_{t=1}^{k_2}\{b_t\cdot \widehat{q}_{2v}\times\widehat{q}_{2v}(\til{V}^-_2)\}$. For concreteness suppose that $\til{\ell}$ is not an element of $\bigcup\limits_{s=1}^{k_1}\{a_s\cdot \widehat{q}_{1v}\times\widehat{q}_{1v}(\til{V}^-_1)\}$. If $\til{\ell}_1$ is its realisation in $\til{G}_1$, then by using Lemma \ref{dip}. we see that $\til{\ell}_1$ is not asymptotic to any lift of any element of  $\mathcal{WL}^\pm_1$. Let $\ell_1$ denote projection of $\til{\ell}_1$ to $G_1$. By taking longer generic leaf segments (thereby reducing $V^-_1$) and increasing $M_0'$ if necessary, we get $\ell_1\notin V^-_1$. 
Since $|\sigma_{1j}|\to \infty$ in $G_1$, we might as well assume that $\sigma_{1j}$ are circuits. 
This  implies that, after passing to a subsequence if necessary,
we have $\sigma_{1j} \notin {V}^-_1$, since $V^+_1$ is an open set. Hence $f_{1\#}^{M'_0} (\sigma_{1j}) \notin {V}^-_1$ because $f_{1\#}^{-1}({V}^-_1) \subset {V}^-_1$. 

By our construction of the open sets we have that $f^m_{1\#}(\ell_1)\in V^+_1 $ for all $m \geq M'_0$. Since  ${V}^+_1$ is open, there exists some $J > 0$ such that $f^{M'_0}_{1\#}(\sigma_{1j}) \in {V}^+_1$ for every $j \geq J$. 

Finally we apply Lemma \ref{expgrowth} to the paths $f_{1\#}^{M'_0}(\sigma_{1j})$. Choose a sequence of real numbers $A_j \to \infty$   to conclude (Lemma \ref{expgrowth}) that for all sufficiently large $j$, $|{f_1}_\#^{n_j}(\sigma_{1j})| \geq A_j |\sigma_{1j}|$. This  implies that the ratio $|{f_1}_\#^{n_j}(\sigma_{1j})| / |\sigma_{1j}| \to \infty$ contradicting our choice of $\sigma_{1j}$'s in Equation \ref{eq}. 
This final contradiction proves the Proposition.
\end{proof}

\begin{rmk}
{\rm 
 In the setting when all  $\phi_i$'s are fully irreducible, the argument after Equation \ref{eq} in the proof of Proposition \ref{3of4} can be made simpler. The limiting lines $\ell_i$ for $i = 1,2$ will be either attracted to $\Lambda^+_i$ or belong to $\mathcal{WL}^-_i$.  Our choice of attracting neighbourhoods ensures that   $\ell_i\notin\mathcal{WL}^-_i$ for at least one $i$. After this one can proceed with the choice of the $A_j$'s 
 as in the proof to get the final contradiction. }
\end{rmk}

\begin{cor}[All but one stretch]\label{cor-stretch}
Given the standard setup, let $v$ be a vertex of $\GG$ and let $e_1, e_2, \cdots, e_k$ be all the oriented edges in $\GG$ which have $v$ as the initial vertex. If $\Phi_1, \Phi_2, \cdots, \Phi_k$ are hyperbolic, rotationless automorphisms  associated to these edges that are independent in $\pi_1(R_v)$, then there exist constants $M_v , L_v >0$ such that  for every geodesic segment $\til{\tau}$  in $\til{R}_v$ of length greater than $L_v$,  at least 2k-1 of the numbers $|(\til{f}_{i\#}^{\pm m}(\til{\tau}_i)|$ are greater than $2 |\til{\tau}| $ for every $m\geq M_v$. Here $\til{\tau}_i$ is the realisation of $\til{\tau}$ in $\til{G}_i$ and $1\leq i \leq k$.
\end{cor}

\begin{proof}
 We choose our constant $A$ as we did in proof of Proposition \ref{3of4} by varying over all the indices involved. We similarly choose attracting neighbourhoods and use Lemma \ref{dne} to get open sets which satisfy the conditions
 of Proposition \ref{3of4} for each pair of elements $\phi_i, \phi_j$ where $1\leq  i\neq j \leq k$. Thus, for each $\phi_i$
  we obtain $2k-1$ open sets. The intersection of  these 
  $2k-1$ open sets is an open set, which we denote by $V^+_i$. We do this for each $1 \leq i \leq k$ and also for the inverses of $\phi_i$. In the process, we get a collection of open sets $V^+_i, V^-_i$ which simultaneously satisfy the conclusions of Lemma \ref{dne} for each pair  $\phi_i, \phi_j$ where $i\neq j$.  Now use these open sets and apply Lemma \ref{WAT} to obtain a constant $M_0$ as in proof of Proposition \ref{3of4}. 
 
We claim that there exists some constant $M_v > M_0, L_v >0$ such that  for every geodesic segment $\til{\tau}$  in $\til{R}_v$ of length greater than $L_v$,  at least $2k-1$ of the numbers $|(\til{f}_{i\#}^{\pm m}(\til{\tau}_i)|$ are greater than $2 |\til{\tau}| $ for every $m\geq M_v$. 
Suppose not. Then there exists a sequence of positive integers $n_j \to \infty$ and paths $\til{\sigma}_j \in \til{R}_v$ with $|\til{\sigma}_j|\to \infty$, such that at least two of the numbers $|\til{f}_{i\#}^{\pm n_j}(\til{\sigma}_{ij})|$ 
($1\leq i \leq k$) are less than $2 |\til{\sigma}_j|$. 

By passing to a subsequence, we can assume without loss of generality  that $|\til{f}_{1\#}^{n_j}(\til{\sigma}_{1j})|, |\til{f}_{2\#}^{n_j}(\til{\sigma}_{2j})| < 2|\til{\sigma}_j|$ for all $j$ and $|\til{\sigma}_{j}| \to \infty$. This violates the 3-of 4 stretch Lemma \ref{3of4}.  This contradiction  completes the proof. 
\end{proof}

 \subsection{Equivalent notion of independence}\label{sec-ind}
Observe that for the proof of the 3-out-of-4 stretch Lemma \ref{3of4}  all that we needed was the  existence of disjoint neighbourhoods  satisfying the conclusions of Lemma \ref{dne}. In this subsection we give some alternate notions of independence of automorphisms that suffice for the purposes of this paper. This section is largely independent of Section \ref{sec-main} and may be omitted on first reading.

\begin{definition}\label{find}
	\emph{(Fixed point independence of automorphisms:)}
	Let $H_1, H_2$ be finite index subgroups of a 
	free group $F$ with indices $k_1, k_2$ respectively.
	Let  $\Phi_1, \Phi_2$ be hyperbolic automorphisms of $H_1, H_2$ respectively. Let ${\{a_i\cdot H_1\}}_{i=1}^{k_1}$ and ${\{b_j\cdot H_2\}}^{k_2}_{j=1}$  be the collections of distinct cosets of $H_1, H_2$ in $F$. We will say  that $\Phi_1, \Phi_2$ are {\emph{fixed point independent in }} $F$  if the following conditions are satisfied:
	\begin{enumerate}
		\item $a_i\cdot \widehat{q}_{1v}(\text{Fix}^\pm_1) \cap a_j\cdot \widehat{q}_{1v}(\text{Fix}^\pm_1) = \emptyset$ for all $1\leq i\neq j \leq k_1$. Similarly $b_i\cdot \widehat{q}_{2v}(\text{Fix}^\pm_2) \cap b_j\cdot \widehat{q}_{2v}(\text{Fix}^\pm_2) = \emptyset$ for all $1\leq i\neq j \leq k_2$.
		\item $a_i\cdot \widehat{q}_{1v}(\text{Fix}^\pm_1) \cap b_j\cdot \widehat{q}_{2v}(\text{Fix}^\pm_2) = \emptyset$ for all $1\leq i\leq k_1, 1\leq j \leq k_2$.
		
	\end{enumerate}

\end{definition}

It immediately follows from this definition that independence in the sense of definition \ref{ind} implies fixed point independence. 
We prove the equivalence of the two definitions via the following lemma. 
For convenience we will address a singular line which is also a generic leaf as \emph{singular leaf}.

\begin{lemma}[Fixed point independence implies disjoint neighbourhoods exist] 
	\label{nae}
	Let $v \in \mathcal{G}$ be any vertex and $e_1, e_2$  be any two edges of $\mathcal{G}$ originating at $v$. If $\Phi_1, \Phi_2$ are fixed point independent in $\pi_1(R_v)$ then disjoint neighbourhoods exist satisfying the conclusions of Lemma \ref{dne}.
	
\end{lemma}

\begin{proof}
	For convenience, we use the variables $\epsilon_1, \epsilon_2 = +, -$
	It suffices to assume that $\Phi_1, \Phi_2$ are fixed point independent in $\pi_1(R_v)$ and produce the required neighbourhoods. Lemma \ref{EG} tells us that the number of attracting and repelling fixed points are finite. 
	Let $\text{Fix}^+_1 = \{x_1, x_2, \dots x_n\}$. Then there exist open sets $\til{U}^+_i$  containing $x_i$ such that $x_j\notin \til{U}^+_i$ if $i\neq j$ (since attracting fixed points are isolated). Similarly define  $\til{U}^-_j$ for repelling fixed points of $\phi_1$. Again by using the fact that these points are isolated we may assume $\til{U}^+_i \cap \til{U}^-_j =\emptyset$ for any $i, j$. Analogously construct pairwise disjoint open sets $\til{V}^+_i, \til{V}^-_j$  corresponding to attracting and repelling fixed points of  $\phi_2$. By taking smaller neighbourhoods if necessary we may assume that $\til{U}^{\epsilon_1}_i \cup \til{V}^{\epsilon_2}_j = \emptyset$ for every $i, j$. By using the finiteness of the index of the subgroups we may shrink these neighbourhoods and use the definition of fixed point independence to get 
	
	\begin{enumerate}
		\item $a_s\cdot \widehat{q}_{1v}(\til{U}^{\epsilon_1}_i) \cap a_t\cdot \widehat{q}_{1v}(\til{U}^{\epsilon_2}_j) = \emptyset$ for all $1\leq s\neq t \leq k_1$ and all $i,j$. Similarly $b_s\cdot \widehat{q}_{2v}(\til{V}^{\epsilon_1}_i) \cap b_t\cdot \widehat{q}_{2v}(\til{V}^{\epsilon_2}_j) = \emptyset$ for all $1\leq s\neq t \leq k_2$ and all $i, j$.
		\item $a_s\cdot \widehat{q}_{1v}(\til{U}^{\epsilon_1}_i) \cap b_t\cdot \widehat{q}_{2v}(\til{V}^{\epsilon_2}_j) = \emptyset$ for all $1\leq s\leq k_1, 1\leq t \leq k_2$ and all $i,j$.
	\end{enumerate}

	Set $\til{A}^{\epsilon_1} = \bigcup_i \til{U}^{\epsilon_1}_i \subset \partial \pi_1{G_i}$ and $\til{B}^{\epsilon_2} = \bigcup_j \til{V}^{\epsilon_2}_j \subset \partial \pi_1(G_2)$. The image of these four sets are pairwise disjoint in $\partial{\til{R}}_v$ and properties (1) and (2)  above naturally extend to the sets $\til{A}^{\epsilon_1}, \til{B}^{\epsilon_2}$. Now consider the open subset $A^{\epsilon_1}_1$ of $\mathcal{B}_1$ given by $\big(\til{A}^{\epsilon_1} \times \til{A}^{\epsilon_1} \setminus \Delta \big) / \mathbb{Z}_2$. Analogously define open sets $B^{\epsilon_1}_2 \subset \til{\mathcal{B}}_2$. 
	Therefore we get four open sets $\til{A}^{\epsilon_1}_1, \til{B}^{\epsilon_2}_2$ whose images in $\til{\mathcal{B}}_v$ are pairwise disjoint. Let $A^{\epsilon_1}_1, B^{\epsilon_2}_2$ denote the images of these open sets in $\mathcal{B}_1, \mathcal{B}_2$ respectively. Then it is immediate that $A^+_1$ and $A^-_1$ are disjoint in $\mathcal{B}_1$. The same is true for $B^+_1, B^-_2$ in $\mathcal{B}_2$.
	
	Lemma \ref{structure} tells us that every attracting (repelling) lamination of $\phi_i$ contains a singular leaf. Therefore every attracting lamination of $\phi_1$ is contained in $A^+_1$ and every repelling lamination of $\phi_1$ is contained in $A^-_1$. An analogous statement is true for $\phi_2$ with the open sets $B^{\epsilon_2}_2$. 
	Also note that since the open set $A^+_1$ is obtained  from attracting neighbourhoods of attracting fixed points of principal lifts of $\phi_1$, we have the property that $\phi_{1\#}(A^+_1) \subset A^+_1$. Similarly $\phi^{-1}_{1\#}(A^-_1) \subset A^-_1$. Analogous statements are true for the image of $B^{\epsilon_2}_2$ under $\phi^{\epsilon_2}_{2\#}$.

	Hence conclusion (i) of Lemma \ref{dne} is satisfied. Pairwise disjointness of images of  open sets $\til{A}^{\epsilon_1}_1, \til{B}^{\epsilon_2}_2$ in $\til{\mathcal{B}}_v$ implies conclusion (ii) of Lemma \ref{dne} is also satisfied.

	Properties (1) and (2) for the open sets $\til{A}^{\epsilon_1}$ and $\til{B}^{\epsilon_1}$ naturally extend under the product maps as follows : 
	
	\begin{enumerate}[label=(\Alph*)]
		\item $a_s\cdot \widehat{q}_{1v} \times \widehat{q}_{1v}(\til{A}^{\epsilon_1}_1) \cap a_t\cdot \widehat{q}_{1v}\times \widehat{q}_{1v}(\til{A}^{\epsilon_2}_1) = \emptyset$ for all $1\leq s\neq t \leq k_1$. Similarly $b_s\cdot \widehat{q}_{2v}\times \widehat{q}_{2v}(\til{B}^{\epsilon_1}_2) \cap b_t\cdot \widehat{q}_{2v}\times \widehat{q}_{2v}(\til{B}^{\epsilon_2}_2) = \emptyset$ for all $1\leq s\neq t \leq k_2$.
		\item $a_s\cdot \widehat{q}_{1v}\times \widehat{q}_{1v}(\til{A}^{\epsilon_1}_1) \cap b_t\cdot \widehat{q}_{2v}\times \widehat{q}_{2v}(\til{B}^{\epsilon_2}_2) = \emptyset$ for all $1\leq s\leq k_1, 1\leq t \leq k_2$.
		
	\end{enumerate}

	Therefore disjoint neighbourhoods exist and properties (A) and (B) above tells us that  conditions (iii) and (iv) are also satisfied from the conclusion of Lemma \ref{dne}. 
\end{proof}

An immediate corollary of this lemma is the following observation.

\begin{cor}
	Fixed point independence of automorphisms and  independence of automorphisms in sense of definition \ref{ind} are equivalent.
	
\end{cor}

\section{Hyperbolic Regluings}\label{sec-main}
Recall (Definition \ref{def-reglue} and the subsequent discussion) that the regluing of a homogeneous graph of roses $\pi: \XX \to \GG$ corresponding to a tuple $\{\phi_e\}$ 
is denoted by $(\XX_{reg},\GG,\pi, \{\phi_e\} )$. Also, recall that 
$(\til\XX_{reg},\TT,\pi_{reg}, \{\til \phi_e\} )$ denotes the universal cover
of such a regluing.  If $\til\XX_{reg}$ is hyperbolic, then
we say that the regluing is hyperbolic (Definition \ref{def-reglue}).
Further recall that  the mid-edge inclusions in  $(\til\XX_{reg},\TT,\pi_{reg}, \{\til \phi_e\} )$ corresponding to lifts of the edge $e$ are given by lifts 
$\til \phi_e$ of $\phi_e$, and hence are $K(e)-$quasi-isometries,
where $K(e)$ depends on $\phi_e$.

We shall say that a regluing $(\XX_{reg},\GG,\pi, \{\phi_e\} )$
corresponding to a tuple $\{\phi_e\}$ is a \emph{rotationless regluing} if each $\phi_e$ is rotationless. 
The following
is an immediate consequence of
Lemma \ref{rotationless}:

\begin{lemma}\label{lem-rednrotless}
Let $\pi:\XX \to \GG$ be a homogeneous graph of roses, and let
$\{\phi_e\}, e \in E(\GG)$ be a tuple of hyperbolic automorphisms. Then there exists $k \in \natls$ such that 
$(\XX_{reg},\GG,\pi, \{\phi_e^k\} )$ is a rotationless regluing.
\end{lemma}

\begin{defn}
We shall say that a  regluing $(\XX_{reg},\GG,\pi, \{\phi_e\} )$ is an \emph{independent} regluing if
\begin{enumerate}
\item Each $\phi_e$ is hyperbolic.
\item For any vertex $v$ and any pair of edges $e_1, e_2$ incident on $v$, $\phi_{e_1}, \phi_{e_2}$ are independent.
\end{enumerate}
\end{defn}

We are now in a position to state the main theorem
of the paper:

\begin{theorem}\label{thm-main} Let $\pi:\XX \to \GG$ be a homogeneous graph of roses, and let
	$\{\phi_e\}, e \in E(\GG)$ be a tuple of hyperbolic automorphisms such that $(\XX_{reg},\GG,\pi, \{\phi_e\} )$ is an \emph{independent} regluing. 
Then there exist $k, n \in \natls$
 such that 
	$(\XX_{reg},\GG,\pi, \{\phi_e^{kn_e}\} )$
	gives a hyperbolic rotationless regluing
	for all $n_e \geq n$. 
\end{theorem}

\begin{rmk}\label{rmk-redn}
Lemma \ref{lem-rednrotless} allows us to choose $k \in \natls$ such that given a tuple $\{\phi_e\}, e \in E(\GG)$ as in Theorem \ref{thm-main}, $\phi_e^k$ is rotationless for all $e$.
Hence, it suffices to  prove Theorem \ref{thm-main}  with
\begin{enumerate}
\item each $\phi_e$  rotationless,
\item $k=1$.
\end{enumerate}
\end{rmk}

The rest of this section is devoted to a proof of Theorem \ref{thm-main} after the reduction given in Remark \ref{rmk-redn}.\\

\noindent {\bf Fixing qi constants:} Given a homogeneous graph
of roses $\pi: \XX \to \GG$, choose a constant $C_1 \geq 1$  such that
for every vertex space $R_v$ and every edge space $G_e$ such that $e$ is incident on $v$, the edge-to-vertex map from $G_e$
to $R_v$ induces a $C_1-$quasi-isometry of universal covers
$\til{R_e} \to \til{R_v}$.

Next, given a tuple  $\{\phi_e\}, e \in E(\GG)$ of rotationless
hyperbolic automorphisms of $G_e$, there exists  a constant $C_2 \geq 1$ such that $\til{\phi_e}: \til{G_e} \to \til{G_e}$
is a $C_2-$quasi-isometry of universal covers.

Also, the number of graphs homotopy equivalent to $G_e$ and carrying a CT map is finite. Hence there exists  a constant $C_3 \geq 1$ such that for any such graph $G_e'$, there exists
a $C_2-$quasi-isometry from $\til{G_e}$ to $\til{G_e'}$ resulting as a lift of a homotopy equivalence between $G_e, \, G_e'$.

Fix $C=C_1C_2C_3$.
All quasi-isometries in the discussion below will turn out
to be $C-$quasi-isometries.\\

\noindent {\bf Subdividing $\GG$:} Given a tuple  $\{\phi_e\}, e \in E(\GG)$ of rotationless
hyperbolic automorphisms of $G_e$ and a tuple  $\{n_e\}$ of positive integers, we now construct a 
subdivision $\GG_{reg}$ of the graph $\GG$ such that
\begin{enumerate}
\item The regluing $(\XX_{reg},\GG,\pi, \{\phi_e^{n_e}\} )$
naturally induces a homogeneous graph of roses structure
$(\XX_{reg},\GG_{reg},\pi_{reg}, \{\phi_e\} )$. Note that the edge labels for the subdivided graph $\GG_{reg}$ are given by 
$\phi_e$ as opposed to $\phi_e^{n_e}$ for $\GG$. However, the total spaces before and after subdivision are homeomorphic by a fiber-preserving homeomorphism. The graphs $\GG$ and $\GG_{reg}$ are clearly homeomorphic as they differ only in terms of simplicial structure.
\item In the universal cover $(\til\XX_{reg},\TT,\pi_{reg}, \{\til \phi_e\} )$, all the edge-to-vertex inclusions are 
$C-$quasi-isometries.
\end{enumerate}

The construction of $\GG_{reg}$ from $\GG$ is now easy to describe. Replace an edge $e$ labeled by $\phi_e^{n_e}$
by a concatenation of $n_e$ edges, each labeled by $\phi_e$
Since the edge-to-vertex inclusions now factor through
$n_e$ edge-to-vertex maps, each given by $\phi_e$, the lifted
edge-to-vertex inclusions in the universal cover 
 $(\til\XX_{reg},\TT,\pi_{reg}, \{\til \phi_e\} )$
are $C-$quasi-isometries.

We note down the output of the above construction:

\begin{lemma}\label{lem-qiconsts}
Given a homogeneous graph
of roses $\pi: \XX \to \GG$, and a tuple  $\{\phi_e\}, e \in E(\GG)$ of rotationless
hyperbolic automorphisms of $G_e$, there exists a constant
$C\geq 1$ such that for any tuple  $\{n_e\}$ of positive integers, there exist
\begin{enumerate}
\item A subdivision $\GG_{reg}$ of $\GG$, where each edge $e$ is replaced by $n_e$ edges, each labeled by $\phi_e$.
\item The regluing $(\XX_{reg},\GG,\pi, \{\phi_e^{n_e}\} )$ is
homeomorphic to $(\XX_{reg},\GG_{reg},\pi_{reg}, \{\phi_e\} )$
by a fiber-preserving homeomorphism.
\item The universal cover $(\til\XX_{reg},\TT,\pi_{reg}, \{\til \phi_e\} )$ is a homogeneous tree of trees satisfying the qi-embedded condition (see Definition \ref{def-tree}). Further, all the quasi-isometry constants of 
$(\til\XX_{reg},\TT,\pi_{reg}, \{\til \phi_e\} )$ are bounded by $C$.
\end{enumerate}
\end{lemma}

\begin{rmk}
The only difference between the  homogeneous tree of trees before and after subdivision lies in the qi constants. Before
subdivision, they are bounded by $C^{n_e}$. After
subdivision, they are bounded by $C$. 
\end{rmk}

Given Lemma \ref{lem-qiconsts}, we would now like to deduce
Theorem \ref{thm-main}  from  the combination theorem 
of Bestvina-Feighn \cite{BF} which says that a tree of hyperbolic spaces is hyperbolic if it satisfies the
hallways flare condition. In the present setup, the 
hallways flare condition of \cite{BF} simplifies using
the results of \cite{mahan-sardar}.
\begin{defn}
Given a homogeneous tree of trees $\pi: \YY \to \TT$, a
\emph{ $k-$qi section} is a $k-$quasi-isometric
embedding  $\sigma : \TT \to \YY$ such that $\pi\circ \sigma$
is the identity map on $\TT$.

A hallway (see Definition \ref{defnofhallway}) $f: [-m,m]\times [0,1] \to \YY$ is said to be a 
 $K-$hallway if 
 \begin{enumerate}
 	\item $\pi\circ f [-m,m]\times \{t\} \to \TT$ is a parametrized geodesic in the base tree $\TT$
 \item  $f: [-m,m]\times \{0\} \to \YY$ and $f: [-m,m]\times \{1\} \to \YY$ are $K-$quasi-isometric 
 sections of the geodesic $\pi\circ f [-m,m]\times \{t\} \to \TT$.
 \end{enumerate}
\end{defn}
Then, in the setup of the present paper,
\cite[Proposition 2.10]{mahan-sardar} gives us the following:

\begin{lemma}\label{lem-qisxn}
For	$\pi: \XX \to \GG$, and a  tuple $\{\phi_e\}, \, e \in E(\GG)$ as in Lemma \ref{lem-qiconsts}
 there exists  $K\geq 1$ such that the following holds: \\
 For any tuple  $\{n_e\}$ of positive integers, and $(\til\XX_{reg},\TT,\pi_{reg}, \{\til \phi_e\} )$ as in 
 Lemma \ref{lem-qiconsts}, and any $z \in \til\XX_{reg}$, there
 exists a $K-$qi section of $\pi_{reg}: \til\XX_{reg}\to \TT,\pi_{reg}$ passing through $z$.
\end{lemma}

Further, \cite[Section 3]{mahan-sardar} shows:

\begin{lemma}\label{mslemma-flare}
Let $K$ and $(\til\XX_{reg},\TT,\pi_{reg}, \{\til \phi_e\} )$ be as in Lemma \ref{lem-qisxn}. Then, $(\til\XX_{reg},\TT,\pi_{reg}, \{\til \phi_e\} )$ is hyperbolic
provided $K-$hallways flare.
\end{lemma}

\noindent {\bf Constructing special hallways:}
A further refinement to Lemma \ref{mslemma-flare} can be extracted from the proof in \cite[Section 3]{mahan-sardar} along the lines of 
\cite{BF-add}. Towards this, we construct a family of special $K-$hallways. Let $f: [-m,m]\times [0,1] \to \til\XX_{reg}$ be a $K-$hallway. 
Further, let $i, i+1 \in  [-m,m]$ be such that $\pi\circ f (\{i\}\times [0,1])$ and $\pi\circ f (\{i+1\}\times [0,1])$ are both interior points of a subdivided edge $e \in E(\GG)$.
We say that $f: [-m,m]\times [0,1] \to \til\XX_{reg}$ is a \emph{special  $K-$hallway} if for all such $i$,  $f (\{i+1\}\times [0,1])$ equals
 $\phi_e(f (\{i\}\times [0,1]))$ (after identifying both vertex spaces with $G_e$). Then  Lemma \ref{mslemma-flare} can be further refined to
 the following:
 
 \begin{lemma}\label{mslemma-flarespecial}
 	Let $K$ and $(\til\XX_{reg},\TT,\pi_{reg}, \{\til \phi_e\} )$ be as in Lemma \ref{lem-qisxn}. Then, $(\til\XX_{reg},\TT,\pi_{reg}, \{\til \phi_e\} )$ is hyperbolic
 	provided special $K-$hallways flare.
 \end{lemma}

In order to prove Theorem \ref{thm-main},  it thus suffices to prove
the following:

\begin{prop}\label{prop-flare}
Let $\pi:\XX \to \GG$ be a homogeneous graph of roses, and let
$\{\phi_e\}, e \in E(\GG)$ be a tuple of hyperbolic rotationless automorphisms such that $(\XX_{reg},\GG,\pi, \{\phi_e\} )$ is an \emph{independent} regluing. 
Then there exist $ n \in \natls$
such that for all $n_e \geq n$, 
the universal cover $(\til\XX_{reg},\TT,\pi_{reg}, \{\til \phi_e\} )$ satisfies the special $K-$hallways flare condition.
Here, $(\til\XX_{reg},\TT,\pi_{reg}, \{\til \phi_e\} )$
is the universal cover of the reglued homogeneous graph of roses 
$(\XX_{reg},\GG_{reg},\pi_{reg}, \{\phi_e\} )$ given by
Lemma \ref{lem-qiconsts}.
\end{prop}

\begin{proof}
The Proposition  will eventually follow from the `All but one stretch'
Corollary \ref{cor-stretch}. For any special $K-$hallway
$f: [-m,m]\times [0,1] \to \til\XX_{reg}$, we shall call 
$\pi\circ f :  [-m,m]\times \{t\} \to \TT$ the base geodesic of the hallway. Further, $\pi\circ f (0,t)$ is called the mid-point
of the base geodesic. Vertices of $\TT$ fall into two classes:
\begin{enumerate}
\item Lifts of $v \in V(\GG)$. These will be called
\emph{original vertices}.
\item Lifts of $v \in V(\GG_{reg})$, where $v$ is a vertex 
at which some $e \in E(\GG)$ is subdivided. These will be called
\emph{subdivision vertices}. Recall that if the regluing map
for $e$ is $\phi_e^{n_e}$, then $e\in \GG$ is subdivided
into $n_e$ edges.
\end{enumerate}
We assume henceforth that all $n_e$ are chosen to be
larger than some $n_0\in \natls$ (to be decided later) so that
any special $K-$hallway that we consider has base geodesic in 
$\TT$ containing at most one original vertex.

By Corollary \ref{cor-stretch},
we can now assume that there exists $n_1 \in \natls$
such that
  any special $K-$hallway with base geodesic 
of length  at least $2n_1$ centered at an original vertex 
satisfies the flaring condition. More precisely, there exists
$A$ such that for all $m \geq n_1$,
any special $K-$hallway 
of girth at least $A$ $f: [-m,m]\times [0,1] \to \til\XX_{reg}$ with $m\geq n_1$ and $\pi_{reg} \circ f
(\{0\} \times [0,1]) = v$, an original vertex satisfies
\begin{equation}\label{eq-flare}
2 l(f(\{ 0 \} \times I)) \leq \, {\rm max} \ \{ l(f(\{ -m \} \times I)),
l(f(\{ m \} \times I)).
\end{equation}

Next, there exists $n_2 \in \natls$ such that for any special $K-$hallway with base geodesic 
of length  at least $2n_2$ and containing only subdivision vertices, Equation \ref{eq-flare}  holds for $m \geq n_2$.
This follows directly from the  hyperbolicity  of the automorphisms $\phi_e$.
We let $N=max\{2n_1, 2n_2\}$. 

We observe now that the concatenation of  two flaring hallways
satisfying Equation \ref{eq-flare} continues to satisfy Equation \ref{eq-flare} provided the overlap of their base
geodesics has length at least $N$. More precisely, let
$[a,b] \subset \TT$ be the base geodesic of a
special $K-$hallway $\HH_1$ and let $[c,d] \subset \TT$
be the base geodesic of a
special $K-$hallway $\HH_2$ such that
\begin{enumerate}
	\item $c \in (a,b)$ and $b \in (c,d)$. Further, $d_\TT (c,b) \geq N$. 
\item $\HH = \HH_1 \cup \HH_2$ is a special $K-$hallway. In particular, over $[c,b] = [a,b] \cap [c,d]$, the qi-sections
(of $[c,b]$) bounding the hallways $\HH_1, \HH_2$ coincide.
\end{enumerate}
Then $\HH$ continues to satisfy Equation \ref{eq-flare}.

It remains to deal with special $K-$hallways whose base geodesics of the form $[a,b]$ contain one original vertex
$v$  such that one of the end-points $a$ or $b$ is at distance
at most $N-1$ from $v$. Thus, the first restriction on $n_0$
(the lower bound on all $n_e$'s) is that $$n_0 \geq 2N.$$
Next, there exists a constant $C_0$ such that for any interval
$[u,v]\subset \TT$ of length at most $N$,  and a special
$K-$hallway $f: [-m,m]\times [0,1] \to \til\XX_{reg}$
with base geodesic $[u,v]$, 
\begin{equation}\label{eq-flare2}
\frac{1}{C_0}	l(f(\{ m \} \times I))\leq 	l(f(\{ -m \} \times I)) \leq \, 
C_0	l(f(\{ m \} \times I)).
\end{equation}

We are finally in a position to determine $n_0$.
Choose $n_0$ such that for all $m \geq n_0-N$, a special 
$K-$hallway with base geodesic of the form $[a,b]$ with
exactly one end-point an original vertex satisfies:

\begin{equation}\label{eq-flare3}
	2C_0 l(f(\{ 0 \} \times I)) \leq \, {\rm max} \ \{ l(f(\{ -m \} \times I)),
	l(f(\{ m \} \times I)).
\end{equation}
It follows from Equation \ref{eq-flare3}, that if $\HH$ is a
special $K-$hallway, whose base geodesic $[a,b]\subset \TT$
of length at least $n_0$ contains exactly one original vertex
$v$  such that  $d(v,a) \leq N$,
then, 
	$$2C_0 l(f(\{ 0 \} \times I)) \leq \, {\rm max} \ \{ l(f(\{ -m \} \times I)), C_0l(f(\{ m \} \times I))\}.$$
In the case that $d(v,b) \leq N$, 
	$$2C_0 l(f(\{ 0 \} \times I)) \leq \, {\rm max} \ \{ C_0l(f(\{ -m \} \times I)), l(f(\{ m \} \times I))\}.$$

In either case (dividing both sides by $C_0$), Equation \ref{eq-flare} is satisfied and we conclude that the special $K-$hallways flare condition is satisfied for $m \geq n_0$.
\end{proof} 

Lemma \ref{mslemma-flare} and Proposition \ref{prop-flare} together complete the proof of Theorem \ref{thm-main}. \hfill $\Box$

As a concluding remark we point out  that the examples of free-by-free hyperbolic groups in \cite{uyanik} and \cite{Gh-18} can be easily reconstructed using Theorem \ref{thm-main}.


\begin{thebibliography}{10}
	
	\bibitem{BF}
	M.~Bestvina and M.~Feighn.
	\newblock A {C}ombination theorem for {N}egatively {C}urved {G}roups.
	\newblock {\em J. Diff. Geom., vol 35}, pages 85--101, 1992.
	
	\bibitem{BFH-97}
	M.~Bestvina, M.~Feighn, and M.~Handel.
	\newblock Laminations, trees, and irreducible automorphisms of free groups.
	\newblock {\em Geom. Funct. Anal.}, 7(2):215--244, 1997.
	
	\bibitem{BF-add}
	Mladen Bestvina and Mark Feighn.
	\newblock Addendum and correction to: ``{A} combination theorem for negatively
	curved groups'' [{J}. {D}ifferential {G}eom. {\bf 35} (1992), no. 1, 85--101;
	{MR}1152226 (93d:53053)].
	\newblock {\em J. Differential Geom.}, 43(4):783--788, 1996.
	
	\bibitem{BFH-00}
	Mladen Bestvina, Mark Feighn, and Michael Handel.
	\newblock The {T}its alternative for {O}ut$(f_n)$. {I}. {D}ynamics of
	exponentially-growing automorphisms.
	\newblock {\em Ann. of Math. (2)}, 151(2):517--623, 2000.
	
	\bibitem{brinkmann}
	P.~Brinkmann.
	\newblock Hyperbolic automorphisms of free groups.
	\newblock {\em Geom. Funct. Anal.}, 10(5):1071--1089, 2000.
	
	\bibitem{minsky-elc2}
	Jeffrey~F. Brock, R.~D. Canary, and Y.~N. Minsky.
	\newblock The {C}lassification of {K}leinian surface groups {II}: {T}he
	{E}nding {L}amination {C}onjecture.
	\newblock {\em preprint}, 2004.
	
	\bibitem{Co-87}
	Daryl Cooper.
	\newblock Automorphisms of free groups have finitely generated fixed point
	sets.
	\newblock {\em J. Algebra}, 111(2):453--456, 1987.
	
	\bibitem{FH-11}
	Mark Feighn and Michael Handel.
	\newblock The recognition theorem for {${\text{Out}}(F_n)$}.
	\newblock {\em Groups Geom. Dyn.}, 5(1):39--106, 2011.
	
	\bibitem{Gh-18}
	P.~Ghosh.
	\newblock {Relative hyperbolicity of free-by-cyclic extensions}.
	\newblock {\em arXiv:1802.08570}, 2018.
	
	\bibitem{Gh-20}
	Pritam Ghosh.
	\newblock Limits of conjugacy classes under iterates of hyperbolic elements of
	{${\text{Out}}(\mathbb{F})$}.
	\newblock {\em Groups Geom. Dyn.}, 14(1):177--211, 2020.
	
	\bibitem{HM-09}
	M.~Handel and Lee Mosher.
	\newblock {Subgroup classification in Out($F_n$)}.
	\newblock {\em arXiv}, 2009.
	
	\bibitem{HM-20}
	Michael Handel and Lee Mosher.
	\newblock Subgroup {D}ecomposition in {${\text{Out}}(F_n)$}.
	\newblock {\em Mem. Amer. Math. Soc.}, 264(1280):0, 2020.
	
	\bibitem{min}
	Honglin Min.
	\newblock Hyperbolic graphs of surface groups.
	\newblock {\em Algebr. Geom. Topol.}, 11(1):449--476, 2011.
	
	\bibitem{mitra-endlam}
	M.~Mitra.
	\newblock Ending {L}aminations for {H}yperbolic {G}roup {E}xtensions.
	\newblock {\em Geom. Funct. Anal. vol.7 No. 2}, pages 379--402, 1997.
	
	\bibitem{mahan-sardar}
	M.~Mj and P.~Sardar.
	\newblock {A Combination Theorem for metric bundles}.
	\newblock {\em Geom. Funct. Anal. 22, no. 6}, pages 1636--1707, 2012.
	
	\bibitem{mahan-sardar2}
	M.~Mj and P.~Sardar.
	\newblock {Propagating quasiconvexity from fibers}.
	\newblock {\em preprint, arXiv:2009.11521}, 2020.
	
	\bibitem{mahan-rafi}
	Mahan Mj and Kasra Rafi.
	\newblock Algebraic ending laminations and quasiconvexity.
	\newblock {\em Algebr. Geom. Topol.}, 18(4):1883--1916, 2018.
	
	\bibitem{mosher-hbh}
	L.~Mosher.
	\newblock A hyperbolic-by-hyperbolic hyperbolic group.
	\newblock {\em Proc. AMS 125}, pages 3447--3455, 1997.
	
	\bibitem{Mos-97}
	Lee Mosher.
	\newblock A hyperbolic-by-hyperbolic hyperbolic group.
	\newblock {\em Proc. Amer. Math. Soc.}, 125(12):3447--3455, 1997.
	
	\bibitem{scott-wall}
	P.~Scott and C.~T.~C. Wall.
	\newblock Topological {M}ethods in {G}roup {T}heory.
	\newblock {\em Homological group theory (C. T. C. Wall, ed.), London Math. Soc.
		Lecture Notes Series, vol. 36, Cambridge Univ. Press}, 1979.
	
	\bibitem{thurstonnotes}
	W.~P. Thurston.
	\newblock The {G}eometry and {T}opology of 3-{M}anifolds.
	\newblock {\em Princeton University Notes}, 1980.
	
	\bibitem{uyanik}
	Caglar Uyanik.
	\newblock Hyperbolic extensions of free groups from atoroidal ping-pong.
	\newblock {\em Algebr. Geom. Topol.}, 19(3):1385--1411, 2019.
	
\end{thebibliography}
\end{document}